\documentclass[twoside,11pt]{amsart}
\usepackage[utf8]{inputenc}
\usepackage[english]{babel}
\usepackage[fixlanguage]{babelbib}
\usepackage{color,graphicx}
\usepackage{amssymb,amsfonts,amsmath,amsthm,amscd, epsfig}
\usepackage{comment}
\usepackage[all]{xy}

\title{Galois Theory for Inverse Semigroup Orthogonal Actions}
\author[Lautenschlaeger and Tamusiunas]{Wesley G. Lautenschlaeger and Thaísa Tamusiunas}
\address{Instituto de Matem\'{a}tica, Universidade Federal do Rio Grande do Sul,  Av. Bento Gon\c{c}alves, 9500, 91509-900. Porto Alegre-RS, Brazil}
\email{wesleyglautenschlager@gmail.com, thaisa.tamusiunas@gmail.com}
\date{}

\usepackage{graphicx}
\usepackage{amssymb}
\usepackage{amsmath}
\usepackage{amsthm}
\usepackage{indentfirst}
\usepackage{tikz}
\usetikzlibrary{cd}

\newcounter{contador}
\numberwithin{contador}{section}

\newtheorem{theorem}[contador]{Theorem}
\newtheorem{prop}[contador]{Proposition}
\newtheorem{lemma}[contador]{Lemma}
\newtheorem{corollary}[contador]{Corollary}

\theoremstyle{definition}
\newtheorem{defi}[contador]{Definition}
\newtheorem{obs}[contador]{Remark}
\newtheorem{exe}[contador]{Example}

\begin{document}

\maketitle

\begin{abstract}
    A Galois correspondence theorem is proved for the case of inverse semigroups acting orthogonally on commutative rings as a consequence of the Galois correspondence theorem for groupoid actions which appears in \cite{paques2018galois}. To this end, we use a classic result of inverse semigroup theory that establishes a one-to-one correspondence between inverse semigroups and inductive groupoids.
\end{abstract}

\vspace{0.5 cm}



\section{Introduction}

S. U. Chase, D. K. Harrison and A. Rosenberg developed in 1965 a Galois theory for finite groups acting on commutative ring extensions, in which they exhibited a generalization of the Fundamental Theorem of Galois Theory \cite{chase1969galois}. The concepts of Galois extension and the Galois correspondence in \cite{chase1969galois} were extended to actions of groupoids in \cite{bagio2012partial} by D. Bagio and A. Paques and in \cite{paques2018galois} by A. Paques and the second author. In \cite{paques2018galois}, if $A|_{A^{\beta}}$ is a commutative groupoid Galois extension with finite groupoid $G$, it was proven that there is a one-to-one correspondence between the wide subgroupoids of $G$ and the subalgebras of $A$ that are separable and $\beta$-strong, where $\beta$ is an action of $G$ on $A$.

In inverse semigroup theory, it is well known that each inverse semigroup is associated, up to isomorphism, to one, and only one, inductive groupoid via the Ehresmann-Schein-Nambooripad Theorem (hereafter, “ESN Theorem”) \cite[Theorem 4.1.8]{lawson1998inverse}. So the purpose of this paper is to use this connection to construct a Galois theory for inverse semigroups acting on commutative rings, based on what has already been done for groupoids. Let $S$ be an inverse semigroup acting on an $R$-algebra $A$ via an action $\beta = (\{E_s\}_{s \in S}, \{ \beta_s : E_{s^{-1}} \rightarrow E_s\}_{s \in S})$. If $A = \bigoplus_{e \in E(S)} E_e$, where $E(S)$ is the set of idempotent elements of $S$, we say that $\beta$ is an \emph{orthogonal} inverse semigroup action. In particular, every group action is orthogonal. We will prove that, given a commutative inverse semigroup Galois extension $A|_{A^{\beta}}$ with finite inverse semigroup $S$, there is a one-to-one correspondence between full inverse subsemigroups of $S$ and separable and $\beta$-strong subalgebras of $A$, where $\beta$ is an orthogonal action of $S$ on $A$. When $S$ is in particular a group, the results that appear in our work coincide with those constructed by Chase, Harrison and Rosenberg in \cite{chase1969galois}.

The paper is organized as follows. In Section 2 we start by fixing notations and recalling for the reader the connection between inverse semigroups and inductive groupoids. In Section 3 we study the connection between the skew inverse semigroup ring and the skew groupoid ring and we define orthogonal inverse semigroup actions. We explore more about the relation between inverse semigroups and inductive groupoids in the respective substructures in Section 4, in which we also analyse the normal inverse semigroups and quotients. In Section 5 we develop a Galois theory, in which we prove a Galois correspondence theorem, that extends the one given in \cite{chase1969galois}. To illustrate the correspondence, we finish the paper constructing an example of a Galois extension of an inverse semigroup of order 28 (which is an inverse subsemigroup of $\mathcal{I}_s(X)$, $X = \{1, 2, 3\}$, of order 34) acting on an algebra generated by six idempotents elements, where we compute all the full inverse semigroups and the corresponding subalgebras.

Throughout, rings and algebras are associative and unital.

\section{Preliminaries}












We will remind the reader of some inverse semigroup and groupoid theory results. We are interested in analysing the relation between inverse semigroups and ordered groupoids. 

Let $G$ be a groupoid and denote by $G_0$ the identities of $G$. Given $g \in G$, the \emph{domain} and the \emph{range} of $g$ will be denoted by $d(g)$ and $r(g)$, respectively. Hence, $d(g) = g^{-1}g$ and $r(g) = gg^{-1}$. For all $g, h \in G$, we write $\exists gh$ whenever the product $gh$ is defined. 

We say that $G$ is \emph{ordered} if there is a partial order $\leq$ in $G$ such that: \begin{itemize}

\item [(OG1)] If $x \leq y$ then $x^{-1} \leq y^{-1}$;

\item [(OG2)] For all $x, y, u, v \in G$ such that $x \leq y$, $u \leq v$, $\exists xu$ and $\exists yv$, we have $xu \leq yv$;

\item [(OG3)] Given $x \in G$ and $e \in G_0$ such that $e \leq d(x)$, there is a unique element $(x | e) \in G$ which satisfies $(x | e) \leq x$ and $d(x | e) = e$.

\item [(OG3*)] Given $x \in G$ and $e \in G_0$ such that $e \leq r(x)$, there is a unique element $(e | x) \in G$ which satisfies $(e | x) \leq x$ and $r(e | x) = e$. \end{itemize}

If $G$ is an ordered groupoid such that $G_0$ is a meet semilattice with respect to $\leq$, we say that $G$ is an \emph{inductive} groupoid.

If $G$ is an ordered groupoid, we say that $H \subseteq G$ is an \emph{ordered subgroupoid} if $H$ is a subgroupoid and if $x \in H$ and $e \in H_0$ are such that $e \leq d(x)$, then $(x | e) \in H$.


\begin{prop} \cite{lawson1998inverse}
Let $(G, \leq)$ be an ordered groupoid. Then the set of identities $G_0$ is an order ideal of $G$.
\end{prop}

Let $S$ be an arbitrary inverse semigroup. We define in $S$ a partial binary operation $\cdot$, named the \emph{restricted product}, by
\begin{align*}
    \exists s \cdot t \Leftrightarrow s^{-1}s = tt^{-1} \text{ and in this case } s \cdot t = st.
\end{align*}

Also, define the relation $\leq$ on $S$ as follows:
\begin{align*}
    s \leq t \Leftrightarrow s = tf, \text{ for some } f \in E(S),
\end{align*}
where $E(S)$ is the set of idempotent elements of $S$.

By \cite[Chapter 3]{lawson1998inverse}, the relation $\leq$ defines a partial order on $S$, called the \emph{natural partial order on $S$}.

Denote $\mathbb{G}(S) := (S, \cdot, \leq)$, that is, the inverse semigroup $S$ equipped with the restricted product and the natural partial order. We say that $\mathbb{G}(S)$ is the \emph{groupoid associated to the inverse semigroup $S$}. One can show that $\mathbb{G}(S)$ is an inductive groupoid for all inverse semigroup $S$ just by taking $(x|e) = xe$ and $(e|x) = ex$.

Given an inductive groupoid $(G, \cdot, \leq)$, we define the \emph{pseudoproduct} $\star$ of $x, y \in G$ by
\begin{align*}
    x \star y = (x | e) \cdot (e | y),
\end{align*}
where $e = d(x) \wedge r(y)$ is the greatest lower bound of $d(x)$ and $r(y)$. That is well defined because $G_0$ is a meet semilattice.

Denote by $\mathbb{S}(G) := (G, \star)$.

\begin{prop} \cite[Proposition 4.1.7]{lawson1998inverse}
Let $G$ be an inductive groupoid.

(i) $\mathbb{S}(G)$ is an inverse semigroup;

(ii) $\mathbb{G}(\mathbb{S}(G)) = G$;

(iii) For any inverse semigroup $S$, $\mathbb{S}(\mathbb{G}(S)) = S$.
\end{prop}

In these cases the partial orders of the inverse semigroup and of the inductive groupoid are the same.

\begin{defi}
A functor $\varphi : G \to H$ between inductive groupoids is called an \emph{inductive functor} if $g_1 \leq_G g_2$ implides $\varphi(g_1) \leq_H \varphi(g_2)$ and for all $e,f \in G_0$, $\varphi(e \wedge f) = \varphi(e) \wedge \varphi(f)$.

A map $\psi : S \to T$ between inverse semigroups is called an \emph{inverse semigroup homomorphism} if $\psi(st) = \psi(s)\psi(t)$.
\end{defi}

The maps $\mathbb{S}$ and $\mathbb{G}$ can also be applied in the morphisms of inverse semigroups and inductive groupoids as follows: if $\psi : S \to T$ is an inverse semigroup homomorphism, then $\mathbb{G}(\psi) : \mathbb{G}(S) \to \mathbb{S}(T)$, where $\mathbb{G}(\psi)(g) = \psi(g)$ is an inductive functor. Reciprocally, if $\varphi : G \to H$ is an inductive functor, then $\mathbb{S}(\varphi) : \mathbb{S}(G) \to \mathbb{S}(H)$ is an inverse semigroup homomorphism, where $\mathbb{S}(\varphi)(s) = \varphi(s)$. So we can consider $\mathbb{S}$ and $\mathbb{G}$ functors. This yields the ESN Theorem.

\begin{theorem}[Ehresmann-Schein-Nambooripad] \cite[Theorem 4.1.8]{lawson1998inverse}
The category of inverse semigroups and homomorphisms is isomorphic to the category of inductive groupoids and inductive functors.
\end{theorem}

\section{Actions of inverse semigroups and inductive actions of inductive groupoids}

We start this section remembering the definition of inverse semigroup action and some of its consequences. The reader can found more about it in \cite{exel2010actions}.

Let $S$ be an inverse semigroup and $A$ an $R$-algebra. We say that $\beta = (\{E_s\}_{s \in S}, \{ \beta_s : E_{s^{-1}} \rightarrow E_s\}_{s \in S})$ is an \emph{action of $S$ on $A$} if $E_s \triangleleft A$, $\beta_s$ is an isomorphism of $R$-algebras, for all $s \in S$, and
\begin{itemize}
\item[(i)] $\beta_s(\beta_t(a)) = \beta_{st}(a)$, for all $s,t \in S$ and $a \in E_{t^{-1}}$;

\item[(ii)] If $S$ is unitary with identity $1_S$, then $E_{1_S} = A$.
\end{itemize}

\begin{prop} \cite{exel2010actions} \label{propconsacoes} Let $\beta$ be an action of $S$ on $A$ and $s, t \in S$. Then:

\emph{(i)} $\beta_{1_S} = \text{Id}_A$;

\emph{(ii)} $\beta_{s^{-1}} = \beta_s^{-1}$;

\emph{(iii)} $\beta_s(E_t \cap E_{s^{-1}}) = E_{st}$;

\emph{(iv)} If $f \in E(S)$, then $\beta_f = Id_{E_f}$;

\emph{(v)} $E_{st} \subseteq E_s$.
\end{prop} 

From (v) above we can notice that
\begin{align*}
    E_s = E_{ss^{-1}s} \subseteq E_{ss^{-1}} \subset E_s,
\end{align*}
that is, $E_s = E_{ss^{-1}}$. Hence the domains of the isomorphisms of an inverse semigroup action are determined by the idempotent elements. 

Denote by $\mathcal{I}_{s}(A)$ the semigroup of partial isomorphisms of $A$, that is,
\begin{align*}
    \mathcal{I}_s(A) = \{ \varphi : B \to C : B,C \triangleleft A \text{ and } \varphi \text{ is an isomorphism} \}.
\end{align*}

The operation in $\mathcal{I}_s(A)$ is the composition between maps. It is clear that an action of an inverse semigroup $S$ on $A$ is precisely an inverse semigroup homomorphism $\beta : S \to \mathcal{I}_s(A)$, where $\beta_s = \beta(s)$.

We also remember the definition of groupoid action. According to \cite{bagio2012partial}, given a groupoid $G$ and an $R$-algebra $A$, $\beta = (\{E_g\}_{g \in G}, \{ \beta_g : E_{g^{-1}} \rightarrow E_g\}_{g \in G})$ is said to be an \emph{action of $G$ on $A$} if $E_g = E_{r(g)} \triangleleft A$, $\beta_g$ is an isomorphism of $R$-algebras, for all $g \in G$, and

\begin{itemize}
\item[(i)] $\beta_e = \text{Id}_{E_e}$, for all $e \in G_0$;

\item[(ii)] $\beta_g(\beta_h(a)) = \beta_{gh}(a)$, for all $(g,h) \in G^2$ and $a \in E_{h^{-1}} = E_{(gh)^{-1}}$.\end{itemize}

If $G$ is ordered, we say that $\beta$ is an \emph{ordered action} if $g \leq h$ implies $E_g \subseteq E_h$ and $\beta_g = \beta_h|_{E_{g^{-1}}}$, for $g,h \in G$.

\begin{exe} \label{ex32}
Let $G = \{ r(x), d(x), x, x^{-1}, r(z), z \}$ be the ordered groupoid with relations $z = z^{-1}$, $x \leq z$. We have that $x^{-1} \leq z$ and $r(x), d(x) \leq r(z)$. Then $G$ is an ordered groupoid which is not inductive, since $r(x) \wedge d(x)$ does not exist.

Take a commutative ring $R$ with unity $1_R$ and consider $A = \bigoplus\limits_{i = 1}^4 Re_i$, where the $e_i$'s are orthogonal idempotents which sum is $1_A$. Define \begin{center}$E_z = E_{r(z)} = A$, \quad  $E_x = E_{r(x)} = Re_1 \oplus Re_3$\end{center} and let $\beta_z : E_z \rightarrow E_z$ be such that \begin{center}$\beta_z(ae_1 + be_2 + ce_3 + de_4) = be_1 + ae_2 + de_3 + ce_4$, $a, b, c, d \in R$.\end{center} With these definitions, we already have an ordered groupoid action of $G$ on $A$. In fact, since $x, x^{-1} \leq z$, we have that $E_x, E_{x^{-1}} \subseteq E_z$ and $\beta_x = \beta_z|_{E_{x^{-1}}}$, $\beta_{x^{-1}} = \beta_z|_{E_x}$, from where it follows that $E_{x^{-1}} = Re_2 \oplus Re_4$, $\beta_x(ae_1 + be_3) = ae_2 + be_4$ and $\beta_{x^{-1}}(ae_2 + be_4) = ae_1 + be_3$. Then $\beta$ is an ordered action of $G$ on $A$.
\end{exe}

Denoting by $\mathcal{I}_g(A) = \mathbb{G}(\mathcal{I}_s(A))$, we have that a groupoid action of $G$ on $A$ is precisely a functor $\beta : G \to \mathcal{I}_g(A)$. The product $\varphi \cdot \psi$ in $\mathcal{I}_g(A)$ of two isomorphisms $\varphi : B \to C$, $\psi : B' \to C'$ is defined when $C' = B$ and in this case $\varphi \cdot \psi = \varphi \circ \psi$, the usual composition.

\begin{defi}
An inductive action of an inductive groupoid $G$ on $A$ is an inductive functor $\beta : G \to \mathcal{I}_g(A)$. That is, an ordered groupoid action $\beta = (\{ E_g \}_{g \in G}, \{ \beta_g \}_{g \in G})$ such that $E_{e \wedge f} = E_e \cap E_f$.
\end{defi}

The next result is a direct consequence of the ESN Theorem.

\begin{corollary}
There is a one-to-one correspondence between inverse semigroup actions and inductive groupoid inductive actions induced by the functors $\mathbb{S}$ and $\mathbb{G}$.
\end{corollary}

We usually will forget the functors $\mathbb{S}$ and $\mathbb{G}$ when considering actions of inverse semigroups and inductive actions of inductive groupoids. Thus, if $\beta : S \to \mathcal{I}_s(A)$ is an inverse semigroup action, we will denote $\mathbb{G}(\beta) : \mathbb{G}(S) \to \mathcal{I}_g(A)$ simply as $\beta$.

From now on, assume that $R$ is a ring, $A$ is an $R$-algebra, $S$ is an inverse semigroup and $\beta = (\{ E_s \}_{s \in S}, \{ \beta_s : E_{s^{-1}} \rightarrow E_s \}_{s \in S})$ is an action of $S$ on $A$ where every $E_s$ is a unital $R$-algebra with unity $1_s$. Also assume that $E_s = 1_sA = A1_s$, that is, every $1_s$ is a central idempotent of $A$.

In order to relate the groupoid Galois Theory developed in \cite{paques2018galois} with the inverse semigroup Galois theory, we introduce a special type of action.

\begin{defi}
Let $S$ be an inverse semigroup acting on an $R$-algebra $A$ via action $\beta$. If
\begin{align} \label{hip1}
    A = \bigoplus_{e \in E(S)} E_e,
\end{align}
we say that $\beta$ is an \emph{orthogonal inverse semigroup action}. 

Orthogonal groupoid actions and orthogonal inductive groupoid actions are defined analogously. 
\end{defi}

\begin{obs}
If $S$ has identity $1_S$, (\ref{hip1}) holds if, and only if, $E_e = 0$ for all $e \in E(S)$ such that $e \neq 1_S$, since $E_{1_S} = A$. In particular, every group action is an orthogonal group action.

The condition (\ref{hip1}) tells us more about the action $\beta$. Since $E_{s^{-1}} = E_{s^{-1}s}$, $E_t = E_{tt^{-1}}$ and $s^{-1}s, tt^{-1} \in E(S)$, we have
\begin{align*}
  E_{s^{-1}} \cap E_t = \begin{cases} 0, \text{ if } s^{-1}s \neq tt^{-1} \\ E_{s^{-1}} = E_t, \text{ c.c.} \end{cases}  
\end{align*}

Hence $E_{(st)^{-1}} = \beta_{t^{-1}}(E_{s^{-1}} \cap E_t) = 0$, whenever $s^{-1}s \neq tt^{-1}$. Since $\beta_{st}$ is an isomorphism, $E_{st} = 0$ and $\beta_{st}$ is the zero isomorphism whenever $s^{-1}s \neq tt^{-1}$.

Besides that, notice that $s \leq t$ implies $E_s = E_t\delta_{s,t}$. In fact, there is $e \in E(S)$ such that $s = et$. Moreover, $s \leq t$ implies that $s^{-1} \leq t^{-1}$, from where $E_s \subseteq E_t = E_{tt^{-1}}$. But $E_s = E_{et} \subseteq E_e$. Therefore $E_s \subseteq E_{tt^{-1}} \cap E_e$. Since $tt^{-1}, e \in E(S)$, 
\begin{align*}
    E_s = E_{et} = \begin{cases} 0, \text{ if } tt^{-1} \neq e \\ E_{et}, \text{ c.c.}
\end{cases}
\end{align*}

If $e = tt^{-1}$, then $s = et = tt^{-1}t = t$. Thus, $s \leq t$ implies that $E_s = \delta_{s,t}E_t$.
\end{obs}

Let $\max E(S)$ (resp. $\max S$) be the set of maximal elements of $E(S)$ (resp. $S$) with respect to the natural partial order. The remarks above yield the next result. 

\begin{prop}
Orthogonal inverse semigroup actions of $S$ on $A$ are precisely the homomorphisms $\beta : S \to \mathcal{I}_s(A)$ such that $E_e = \{ 0 \}$ for all $e \notin \max E(S)$. In particular, there is a one-to-one correspondence between orthogonal inverse semigroup actions and orthogonal inductive groupoid actions.
\end{prop}

\begin{exe} \label{exgeral}
Let $R$ be a commutative ring, $n > 0$ be an integer and $A$ be an $R$-algebra  that splits as
\begin{align*}
    A = \bigoplus_{i=1}^{nk} Re_i,
\end{align*}
where $n = \binom{m}{k}$, for some positive integers $m,k$, and the $e_i$'s are orthogonal idempotents such that $\sum_i e_i = 1_A$. 

We will construct an inverse semigroup $S$ and an orthogonal action $\beta$ of $S$ on $A$.

Let $X = \{ 1, \ldots, m \}$ and consider $\mathcal{I}_s(X)$. Denote by $S = \{ s \in \mathcal{I}_s(X) : \# \text{dom}(s) \leq k \}$.  Enumerate the set of $k$-subsets of the elements of $X$ as $M_1, M_2, \ldots, M_{n}$. Then $|M_i| = k$, for all $1 \leq i \leq n$. Thus, if $s \in \max S$, then dom$(s) = M_i$, for some $1 \leq i \leq n$.

We can associate the elements $e_i$ with the $k$-combinations $M_j$ as it follows: 
\begin{align*}
    \{ e_1, \ldots, e_k \} & \leftrightarrow M_1 \\
    \{ e_{k+1}, \ldots, e_{2k} \} & \leftrightarrow M_2 \\
    & \vdots \\
    \{ e_{(n-1)k + 1}, \ldots, e_{nk} \} &  \leftrightarrow M_n
\end{align*}
increasing ordered, that is, $M_i = \{ m_{i_1}, \ldots , m_{i_k} \}$ with $m_{i_j} \leq m_{i_{j+1}}. $ We can, then, write $A_{M_i} = \bigoplus_{j = 1}^k Re_{(i - 1)k + j}$. 

For $s \in \max S$, we have that dom$(s) = M_i$ and Im$(s) = M_j$. We have that $e_{(i-1)k + p}$ is associated to a $m_{i_p} \in M_i$. We know that $s(m_{i_p}) = m_{j_q} \in M_j$. Denote by $\overline{s}(p) = q$. We can define the map $\tilde{s} : \{ (i-1)k + 1, \ldots , ik \} \to \{ (j-1)k+1, \ldots , jk \}$ by $\tilde{s}((i-1)k + p) = (j-1)k + q = (j-1)k + \overline{s}(p)$, uniquelly defined by $s$.

Define $E_s = A_{M_j}$, $E_{s^{-1}} = A_{M_i}$,
\begin{align*}
    \beta_s\left ( \sum_{l = 1}^k r_le_{(i-1)k + l} \right ) = \sum_{l = 1}^k r_le_{(j-1)k + \overline{s}(l)}.
\end{align*}

For $s \notin \max S$, define $E_s = E_{s^{-1}} = \{ 0 \}$ and $\beta_s$ the zero isomorphism. Then $\beta = (\{ E_s \}_{s \in S}, \{ \beta_s \}_{s \in S})$ is an orthogonal action of $S$ on $A$. In fact, let $s, t \in S$ be such that dom$(s) = M_g$, Im$(s) = M_h$, dom$(t) = M_i$, Im$(t) = M_j$. If $j \neq g$, then $\beta_s \circ \beta t = 0 = \beta_{st}$. So assume that Im$(t) = M_g$. Thus,
\begin{align*}
    \beta_s\left (\beta_t\left ( \sum_{l = 1}^k r_le_{(i-1)k + l} \right )\right )  & = \beta_s\left ( \sum_{l = 1}^k r_le_{(g-1)k + \overline{t}(l)} \right ) \\
    & = \sum_{l = 1}^k r_le_{(h-1)k + \overline{s}(\overline{t}(l))} \\
    & = \sum_{l = 1}^k r_le_{(h-1)k + \overline{st}(l)} \\
    & = \beta_{st} \left ( \sum_{l = 1}^k r_le_{(g-1)k + l} \right ).
\end{align*}

By our construction, $S$ is unitary if, and only if, $n=1$ and $m = k$. In this case, $A = \bigoplus_{i = 1}^m Re_i$ and $M_1 \leftrightarrow \{ e_1, \ldots, e_m \}$, so that $E_{1_S} = A_{M_1} = A$ and $\tilde{1_S}$ is the identity map in $\{ 1, \ldots, m \}$, so $\beta_{1_S} = \text{Id}_A$, as required.

The fact that $\beta$ is orthogonal is immediate. In Example \ref{exfinal} we will construct an orthogonal action using these steps with $m = n = 3, k = 2$.
\end{exe}

The next result shows a relation between orthogonal actions and inductive groupoid actions.

\begin{prop}
Let $G$ be an inductive groupoid and $\beta : G \to \mathcal{I}_g(A)$ an orthogonal ordered groupoid action. Then $\beta$ is an orthogonal inductive groupoid action.
\end{prop}
\begin{proof}
Let $e,f \in G_0$. If $e = f$, then $E_e \cap E_f = E_e \cap E_e = E_e = E_{e \wedge f}$. If $e \neq f$, then $e \wedge f \leq e$ and $e \wedge f \leq f$, so we have $E_{e \wedge f} = E_{f}\delta_{e \wedge f, f} = E_{e}\delta_{e \wedge f, e}$, because $\beta$ is orthogonal. Since $e \neq f$ it is not possible that both of the $\delta$'s above are simultaneosly non zero. So $E_{e \wedge f} = \{ 0 \} = E_e \cap E_f$. 
\end{proof}

Also, recall the definitions of crossed products.

According to \cite{exel2010actions}, given an inverse semigroup $S$ and an action $\beta$ of $S$ on an $R$-algebra $A$, we define
\begin{align*}
    L = L(A,S,\beta) = \left \{ \sum\limits_{s \in S}^{\text{finite}} au_s : a \in E_s \right \} = \bigoplus\limits_{s \in S} E_su_s,
\end{align*}
where every $u_s$ is a symbol. The addition is usual and the product is given by
\begin{align*}
    (au_s)(bu_t) = \beta_s(\beta_{s^{-1}}(a)b)u_{st},
\end{align*}
linearly extended.

The crossed product $L$ is not always associative. For the next definition, we will assume that $L$ is associative and thereby an $R$-algebra. A sufficient condition for associativity is presented in \cite[Theorem 3.4]{exel2010actions}. 

In the same conditions of the above definition, consider $N = \langle au_s - au_t : s \leq t, a \in E_s \rangle \triangleleft L$. We define the skew inverse semigroup ring $A \ltimes_{\beta} S$ as
\begin{align*}
    A \ltimes_{\beta} S = \displaystyle\frac{L}{N}.
\end{align*}

\begin{prop} \label{propisoskewsemi}
If $\beta : S \to \mathcal{I}_s(A)$ is an orthogonal inverse semigroup action, then $L(A,S,\beta) = A \ltimes_\beta S$.
\end{prop}
\begin{proof}
It follows from $N = 0$. In fact, given $s, t \in S$ such that $s \leq t$, we have that $E_{s} = E_t\delta_{s,t}$. If $s \neq t$, then $au_s - au_t = 0u_s - 0u_t = 0$, since $a \in E_s$ implies $a = 0$. If $s = t$, then $au_s - au_t = au_t - au_t = 0$, for all $a \in E_s = E_t$.
\end{proof}

Let $G$ be a groupoid and $\beta$ be an action of $G$ on an $R$-algebra $A$. According to \cite{bagio2010partial}, the skew groupoid ring $A \ltimes_{\beta} G$ is defined as
\begin{align*}
    A \ltimes_{\beta} G = \left \{ \sum\limits_{s \in S}^{finite} au_g : a \in E_g \right \} = \bigoplus\limits_{g \in S} E_gu_g,
\end{align*}
where the $u_g$'s are symbols. The addition is usual and the product is given by
\begin{align*}
    (au_g)(bu_h) = \begin{cases} \beta_g(\beta_{g^{-1}}(a)b)u_{gh}, \text{ if } (g,h) \in G^2 \\ 0, \text{ c.c.} \end{cases}
\end{align*}
and linearly extended. In the same way of the inverse semigroup case, $A \ltimes_\beta G$ is not always associative. A sufficient condition to associativity is given in \cite[Proposition 3.1]{bagio2010partial}.

Now we give a definition of crossed product for ordered groupoids.

\begin{defi}
Let $G$ be an ordered groupoid and $\beta$ an ordered groupoid action of $G$ on $A$. Assume that $A \ltimes_\beta G$ is associative. Define $N = \langle au_g - au_h : a \in E_g, g \leq h \rangle \triangleleft A \ltimes_{\beta} G$. We define the skew ordered groupoid ring $A \ltimes_{\beta}^o G$ as
\begin{align*}
    A \ltimes_{\beta}^o G = \displaystyle\frac{A \ltimes_{\beta} G}{N}.
\end{align*}
\end{defi}

It is easy to see that just as in the case of Proposition \ref{propisoskewsemi}, if $\beta$ is orthogonal, then $A \ltimes_\beta^o G = A \ltimes_\beta G$.

\begin{corollary} \label{corolarioprodutoscruzados}
\emph{(i)} Let $S$ be an inverse semigroup and $\beta$ be an orthogonal action of $S$ on $A$. Then $A \ltimes_\beta S \simeq A \ltimes_\beta \mathbb{G}(S)$.

\emph{(ii)} Let $G$ be an inductive groupoid and $\beta$ be an orthogonal inductive groupoid action of $G$ on $A$. Then $A \ltimes_\beta G \simeq A \ltimes_\beta \mathbb{S}(G)$.
\end{corollary}

\begin{proof}
Consider the natural maps
\begin{align*}
    \varphi : A \ltimes_\beta S & \rightarrow A \ltimes_\beta \mathbb{G}(S) \\
    au_s & \mapsto au_s
\end{align*}
and
\begin{align*}
    \psi : A \ltimes_\beta G & \rightarrow A \ltimes_\beta \mathbb{S}(G) \\
    au_g & \mapsto au_g.
\end{align*}

By the comments above, we do not need to worry about equivalence classes. These maps are obviously $R$-module isomorphisms. We will show that $\varphi$ preserves product, the other case is similar. Take $au_s, bu_t \in A \ltimes_\beta S$. Then
\begin{align*}
    \varphi((au_s)(bu_t)) & = \varphi(\beta_s(\beta_{s^{-1}}(a)b)u_{st}).
\end{align*}

However
\begin{align*}
    \beta_s(\beta_{s^{-1}}(a)b) = \begin{cases}
        \beta_s(\beta_{s^{-1}}(a)b), \text{ if } s^{-1}s = tt^{-1} \\
        0, \text{ c.c.}
    \end{cases}
\end{align*}
since $\beta_{s^{-1}}(a)b \in E_{s^{-1}} \cap E_t = E_{s^{-1}s} \cap E_{tt^{-1}}$. Hence
\begin{align*}
    \varphi((au_s)(bu_t)) = \begin{cases}
    \beta_s(\beta_{s^{-1}}(a)b)u_{s \cdot t}, \text{ if } d(s) = r(t) \text{ in } \mathbb{G}(S) \\
    0, \text{ c.c.}
    \end{cases}
\end{align*}

Observe that this is exactly the product in $A \ltimes_\beta \mathbb{G}(S)$. Therefore, \[ 
    \varphi((au_s)(bu_t)) = \varphi(au_s)\varphi(bu_t). \qedhere \]
\end{proof}

Notice that this corollary tells us that when $\beta$ is orthogonal, $A \ltimes_\beta S$ is associative if and only if $A \ltimes_\beta \mathbb{G}(S)$ is associative.

\section{Normal substructures and quotients}

In this section we will understand the relations between the normal inverse subsemigroups of an inverse semigroup $S$ and the normal ordered subgroupoids of $\mathbb{G}(S)$. After this, we will analyze the quotient structures.

An inverse subsemigroup $T$ of $S$ is called \emph{full} if $E(T) = E(S)$. Every full inverse subsemigroup is an order ideal. To see that, just observe that if $t \in T$, $s \in S$, $s \leq t$ and $T$ is a full inverse subsemigroup, we have that $s = ti$, for some $i \in E(S) \subseteq T$, thus $s = ti \in T$. A subgroupoid $H$ of $G$ is called \emph{wide} if $H_0 = G_0$. It is easy to see that a wide ordered subgroupoid is a wide \emph{inductive} subgroupoid.

The ESN Theorem applied in the inclusion maps guarantees that there is a one-to-one correspondence between the full inverse subsemigroups of $S$ and the wide ordered subgroupoids of $\mathbb{G}(S)$.

Let us study the relations between normal inverse subsemigroups and normal ordered subgroupoids. This will allow us to understand the relations between the quotient structures.

\begin{defi}[Normal inverse subsemigroup] \cite{lawson1998inverse}
Let $S$ be an inverse semigroup and $T$ an inverse subsemigroup of $S$. We say that $T$ is \emph{normal} if

(i) $T$ is full;

(ii) For all $s \in S$, $s^{-1}Ts \subseteq T$.
\end{defi}

\begin{defi}[Normal ordered subgroupoid] \label{grpnormal} \cite[Definition 4.1]{alyamani2016fibrations}
Let $G$ be an ordered groupoid and $H$ an ordered subgroupoid of $G$. We say that $H$ is \emph{normal} if

(i) $H$ is wide;

(ii) Given $a, b \in G$ such that $a$ and $b$ have an upper bound, that is, there is $g \in G$ such that $a \leq g$, $b \leq g$, then $a^{-1} \cdot H \cdot b \subseteq H$.
\end{defi}

Observe that if we take a groupoid $G$ without partial order, we can consider $G$ an ordered groupoid by taking the equality as partial order. In this case, Definition \ref{grpnormal} is the same as the standard definition of normal subgroupoid.

In \cite[Example 4.3]{alyamani2016fibrations} it was proven that there is a one-to-one correspondence between the normal inverse subsemigroups of an inverse semigroup $S$ and the normal ordered subgroupoids of $\mathbb{G}(S)$.

Let $(G, \leq)$ be an ordered groupoid and $H$ a normal ordered subgroupoid of $G$. We define the relation $\sim_H$ as, given $a,b \in G$,
\begin{align*}
    a \sim_H b \Leftrightarrow \text{ there are } x, y, u, v \in H & \text{ such that } \exists x \cdot a \cdot y, \exists u \cdot b \cdot v, \\  x \cdot a \cdot y \leq b & \text{ and } u \cdot b \cdot v \leq a.
\end{align*}

By \cite{alyamani2016fibrations} this is an equivalence relation, therefore the quotient $G / H := G/\sim_H$ is well-defined. Define the partial order $\leq_H$ as:
\begin{align*}
    [a]_H \leq_H [b]_H \Leftrightarrow \text{ there are } x, y \in H \text{ such that } \exists x \cdot a \cdot y \text{ and } x \cdot a \cdot y \leq b.
\end{align*}

By \cite[Theorem 4.14]{alyamani2016fibrations} we have that $G/H$ is an ordered groupoid. The operation in $G/H$ is given subsequently. If $d(g) \sim_H r(h)$, then there are $a, b \in H$ such that $r(a) \leq d(g)$, $d(a) = r(h)$, $r(b) \leq r(h)$ e $d(b) = d(g)$. Denoting by $g' = (g | r(a))$ and $h' = (r(b) | h)$, we define
\begin{align*}
    [g]_H \cdot [h]_H = [g' \cdot a \cdot h]_H = [g \cdot b^{-1} \cdot h']_H.
\end{align*}

This definition does not depend on the choice of $a$ and $b$. In fact, for any elements $x, y \in H$ such that $r(x) \leq d(g)$, $d(x) = r(h)$, $r(y) \leq r(h)$ and $d(y) = d(g)$, we have, denoting by $g'' = (g|r(x))$, $h'' = (r(y) | h)$, that
\begin{align*}
    [g' \cdot a \cdot h]_H = [g'' \cdot x \cdot h]_H \text{ and } [g \cdot b^{-1} \cdot h']_H = [g \cdot y^{-1} \cdot h'']_H.
\end{align*}

Now we will translate this definition to the case of inverse semigroups. Given an inverse semigroup $S$ and a normal inverse subsemigroup $T$, we define the relation $\sim_T$ as, given $a, b \in S$,
\begin{align*}
    a \sim_T b \Leftrightarrow \text{ there are } x, y, u, v \in T & \text{ such that } \exists x \cdot a \cdot y, \exists u \cdot b \cdot v, \\ x \cdot a \cdot y \leq b & \text{  and } u \cdot b \cdot v \leq a.
\end{align*}

Notice that we used the restricted product in this definition. The order $\leq_T$ is defined in the same way of the case of inductive groupoids. We denote by $S/T := S/\sim_T$.

A quotient of an inverse semigroup by a normal inverse subsemigroup is always an ordered groupoid \cite[Theorem 3.6]{alyamani2018ordered}, but in general is not an inductive groupoid, that is, an inverse semigroup \cite[Example 3.5]{alyamani2018ordered}. A condition to that case appears in \cite[Proposition 3.9]{alyamani2018ordered}. 

In the case that the set of idempotents of $S/T$ is a meet semilattice, the operation in $S/T$, denoted by $\star$, will be given in terms of the operation in $\mathbb{G}(S)/\mathbb{G}(T)$, that is,
\begin{align*}
    [s]_T \star [t]_T = ([s]_T | [e]_T) \cdot ([e]_T | [t]_T),
\end{align*}
where $[e]_T = [s^{-1}s]_T \wedge [tt^{-1}]_T$.

\begin{theorem} \label{teoquoc}
\emph{(i)} Let $S$ be an inverse semigroup and $T$ a normal inverse subsemigroup of $S$. If $S/T$ is an inverse semigroup, then $\mathbb{G}(S/T) \simeq \mathbb{G}(S)/\mathbb{G}(T)$.

\emph{(ii)} Let $G$ be an inductive groupoid and $H$ a normal ordered subgroupoid of $G$. If $G/H$ is an inductive groupoid, then $\mathbb{S}(G/H) \simeq \mathbb{S}(G)/\mathbb{S}(H)$.
\end{theorem}

\begin{proof}
(i): Denote by $H = \mathbb{G}(T)$. We define the natural map
\begin{align*}
    \varphi : \mathbb{G}(S/T) & \rightarrow \mathbb{G}(S)/H \\
    [s]_T & \mapsto [s]_H.
\end{align*}

\noindent $\bullet$ $\varphi$ is well-defined: take $s,t \in S$ such that $[s]_T = [t]_T$. Since $s \sim_T t$, there are $x,y,u,v \in T$ such that the products are defined and $x \cdot s \cdot y \leq t$ and $u \cdot t \cdot v \leq s$. But that is the same condition of $s \sim_H t$. Hence $[s]_H = [t]_H$.

\noindent $\bullet$  $\varphi$ is an ordered functor: let $s, t \in T$. If $[s^{-1}s]_T = [tt^{-1}]_T$, then $[s^{-1}s]_H = [tt^{-1}]_H$, so the product is defined in the domain and in the image. Now, $\varphi ([s]_T \cdot [t]_T) = \varphi([s' \cdot a \cdot t]_T) = [s' \cdot a \cdot t]_H = [s]_H \cdot [t]_H$, where $a \in H$ is such that $aa^{-1} \leq s^{-1}s, a^{-1}a = tt^{-1}$ and $s' = (s | aa^{-1})$. 

Now take $s,t \in T$ such that $[s]_T \leq_T [t]_T$. That is, there are $x, y \in T$ such that $x \cdot s \cdot y \leq t$. Thus, since $x, y \in H$, we have that $[s]_H \leq_H [t]_H$. 

\noindent $\bullet$  $\varphi$ is surjective: immediate.

\noindent $\bullet$  $\varphi$ is injective: we just need to show that $\ker \varphi = \{ [s]_T \in \mathbb{G}(S/T) : \varphi([s]_T) = [s]_H \in (\mathbb{G}(S)/H)_0 \} = \mathbb{G}(S/T)_0$. Take $[s]_T \in \ker \varphi$. This implies that $[s]_H$ is idempotent, that is, there is $e \in [s]_H$ such that $e$ is idempotent in $\mathbb{G}(S)$. Hence $[s]_H = [e]_H \Rightarrow [s]_T = [e]_T \in \mathbb{G}(S/T)_0$.

\noindent $\bullet$ $\varphi$ is an inductive functor: The meet-preserving follows from $\varphi$ being an ordered groupoid isomorphism. 

(ii): Consider $T = \mathbb{S}(H)$ and the map
\begin{align*}
    \psi : \mathbb{S}(G/H) & \rightarrow \mathbb{S}(G)/T \\
    [s]_H & \mapsto [s]_T.
\end{align*}

The verification that $\psi$ is a well-defined bijection is similar to (i). It only remains to us to show that $\psi$ is a homomorphism of inverse semigroups. Let $g, h \in G$. We have that
\begin{align*}
    \psi([g]_H \star [h]_H) & = \psi(([g]_H | [e]_H) \cdot ([e]_H | [h]_H)),
\end{align*}
where $[e]_H = [g^{-1}g]_H \wedge [hh^{-1}]$.

Denote by $x$ a representative of the coset $([g]_H | [e]_H)$ and by $y$ a representative of the coset $([e]_H | [h]_H)$. Then
\begin{align*}
    \psi([g]_H \star [h]_H) = \psi([x' \cdot a \cdot b]_H) = [x' \cdot a \cdot y]_T,
\end{align*}
where $a \in T$ is such that $aa^{-1} \leq xx^{-1}$, $a^{-1}a = bb^{-1}$ e $x' = (x | aa^{-1})$. Thus \[
    \psi([g]_H \star [h]_H) = [x]_T \cdot [y]_T = ([g]_T | [e]_T ) \cdot ([e]_T | [h]_T) = [g]_T \star [h]_T. \qedhere \]
\end{proof}

Given a (not necessarily ordered) groupoid $G$ and a normal subgroupoid $H$ of $G$, the congruence $\equiv_H$ which defines the quotient groupoid $G/H$ in \cite{paques2018galois} is given by
\begin{align*}
    a \equiv_H b \Leftrightarrow \exists b^{-1} \cdot a \mbox{ and } b^{-1} \cdot a \in H.
\end{align*}

We have that $a \equiv_H b \Rightarrow a \sim_H b$ in the case of ordered groupoids. In fact, let $a, b \in G$ be such that $a \equiv_H b$. Then $\exists b^{-1} \cdot a$ e $b^{-1} \cdot a \in H$. Hence there is $h \in H$ such that $b^{-1} \cdot a = h$. Therefore
\begin{align*}
    a = r(b) \cdot b \cdot h \mbox{ and } b = h^{-1} \cdot a \cdot d(a),
\end{align*}
where $h, h^{-1}, r(b), d(a) \in H$, since $H$ is a normal subgroupoid, therefore wide. Thus $a \leq r(b) \cdot b \cdot h$ and $b \leq h^{-1} \cdot a \cdot d(a)$, from where $a \sim_H b$. 

However the reverse is not always true. But every groupoid can be seen as an ordered groupoid taking the partial order as equality. In this case, the relations $\equiv_H$ and $\sim_H$ are the same. In fact, let $a, b \in G$ be such that $a \sim_H b$. So there are $\exists x, y, u, v \in H$ such that $\exists x \cdot a \cdot y$, $\exists u \cdot b \cdot v$, $x \cdot a \cdot y = b$ and $u \cdot b \cdot v = a$. Notice that
\begin{align*}
    x \cdot a \cdot y = b & \Rightarrow a \cdot y = x^{-1} \cdot b.
\end{align*}

But $H$ is normal in $G$, from where it follows that $H \cdot b = b \cdot H$. Hence there is $z \in H$ such that $x^{-1} \cdot b = b \cdot z$. Thus
\begin{align*}
    a \cdot y = b \cdot z \Rightarrow b^{-1} \cdot a = z \cdot y^{-1} \in H,
\end{align*}
that is, $a \equiv_H b$. We can now see $\sim_H$ as a generalization of $\equiv_H$. In fact, \cite[Proposition 4.6]{alyamani2018ordered} gives us a condition to $\equiv_H$ and $\sim_H$ be equal: $hh^{-1} = h^{-1}h$, for all $h \in H$. That condition includes well-known types of groupoids, like disjoint unions of groups. 

\begin{defi} \begin{itemize} \item[(i)]An inverse semigroup $S$ such that $ss^{-1} = s^{-1}s$, for all $s \in S$, is said to be a \emph{Clifford inverse semigroup}. \item[(ii)] A groupoid $G$ such that $r(g) = d(g)$, for all $g \in G$, is said to be a \emph{Clifford groupoid}.\end{itemize}
\end{defi}


\begin{prop}
Let $T$ be a Clifford normal inverse semigroup of a inverse semigroup $S$. Then $S/T$ is an inverse semigroup.
\end{prop}
\begin{proof}
We have that $H = \mathbb{G}(T)$ is Clifford in $G = \mathbb{G}(S)$. Notice that if $e, f \in G_0$, then $[e]_H \leq_H [f]_H$ if, and only if, there are $g,h \in H$ such that $\exists g \cdot e \cdot h$ and $g \cdot h = g \cdot e \cdot h \leq f$. Since $G_0$ is an order ideal, that means that $gh \in G_0$. But $H$ is Clifford, so $r(g) = d(g) = e = r(h) = d(h)$, from where it follows that $e = r(g) = r(gh) = gh$. So $e \leq f$. 

Reciprocally, if $e \leq f$ in $G_0$, then $\exists e \cdot e \cdot e$ and $e^3 = e \leq f$ in $G$. So $[e] \leq_H [f]$. This means that the orders $\leq_H$ and $\leq$ coincide in $G_0$ and $(G/H)_0$ when $H$ is Clifford. Therefore, since $\sim_H$ is idempotent separating when $H$ is Clifford, we have that $S/T$ is an inverse semigroup.
\end{proof}

\section{Inverse semigroup Galois theory}

Let $R$ be a ring and $A$ an $R$-algebra. We denote by $A^e = A \otimes_R A^{op}$ the enveloping algebra of $A$. Notice that $A$ is a left $A^e$-module via $(x \otimes y) \cdot a = xay$, for all $a, x, y \in A$. We say that $A$ is a \emph{separable} $R$-algebra if it is a left projective $A^e$-module \cite{chase1969galois}.

In this section we assume that $S$ is a finite inverse semigroup, that $A$ is a commutative $R$-algebra and that $\beta$ is orthogonal. We denote by
\begin{align*}
    A^{\beta} = \{ a \in A | \beta_s(a1_{s^{-1}}) = a1_s, \text{ for all } s \in S \},
\end{align*}
the subalgebra of $A$ fixed by $\beta$. 

Define the trace map $tr_\beta : A \rightarrow A$ given by $tr_\beta(a) = \sum_{s \in S} \beta_s(a1_{s^{-1}})$, for all $a \in A$. This map is $A^{\beta}$-linear. We will prove that $tr_\beta(A) \subseteq A^\beta$.

Let $t \in S$ and $a \in A$. Then
\begin{align*}
    \beta_t(tr_\beta(a)1_{t^{-1}}) & = \beta_t \left ( \sum_{s \in S} \beta_s(a1_{s^{-1}})1_{t^{-1}} \right ) \\ & = \sum_{s \in S} \beta_t(\beta_s(a1_{s^{-1}})1_t^{-1}) \\ & = \sum_{t^{-1}t = ss^{-1}} \beta_{ts}(a1_{(ts)^{-1}})1_t \\ & = \sum_{uu^{-1} = tt^{-1}} \beta_u(a1_u)1_t \\ & = \sum_{u \in S} \beta_u(a1_u)1_t \\ & = tr_\beta(a)1_t,
\end{align*}
since $\beta_s(a1_{s^{-1}})1_{t^{-1}} \in E_{s} \cap E_{t^{-1}} = 0$, if $ss^{-1} \neq t^{-1}t$.

Here the orthogonality was necessary. Without it we could not have guaranteed this fact, since in an arbitrary inverse semigroup $S$, given $t \in S$, $tS \neq S$ in general.

Since $tr_\beta(A) \subseteq A^\beta$ we can consider $A^\beta$ as the codomain of the map $tr_\beta$, that is, we can define $tr_\beta : A \to A^\beta$.

\begin{defi}
We say that $A$ is a $\beta$-Galois extension of $A^{\beta}$ if there are $x_i, y_i \in A$, $1 \leq i \leq n$ such that $\sum\limits_{i=1}^n x_i\beta_s(y_i1_{s^{-1}}) = \delta_{e,s}1_e$, for all $s \in S$, $e \in E(S)$. In this case we say that the set $\{ x_i, y_i\}_{1 \leq i \leq n}$ is a system of $\beta$-Galois coordinates of the extension $A|_{A^{\beta}}$.
\end{defi}

\begin{exe}
Let $A$ be an algebra, $S$ be an inverse semigroup and $\beta$ an orthogonal action of $S$ on $A$ as in Example \ref{exgeral}. We have that $A^\beta = R1_A$ and the extension $A|_{A^\beta}$ is $\beta$-Galois with system of $\beta$-Galois coordinates given by $x_i = y_i = e_i$, for all $1 \leq i \leq nk$.
\end{exe}

\begin{exe} \label{ex1}
Let $G$ be a finite group with identity element 1 and $\alpha = (\{ D_g\}_{g \in G}, \{ \alpha_g : D_{g^{-1}} \rightarrow D_g \}_{g \in G})$ a partial action of $G$ on $A$ \cite[Defition 1.2]{exel1998partial}. Assume that every $D_g$ is unital with identity $1_g$ and that $A$ is a partial $\alpha$-Galois extension of $A^{\alpha}$ in the sense of \cite{dokuchaev2007partial}, that is, exist $x_i, y_i$, $1 \leq i \leq n$ such that
\begin{align*}
    \sum_{i = 1}^{n} x_i\alpha_g(y_i1_{g^{-1}}) = \delta_{1,g}1_A, \text{ for all } g \in G.
\end{align*}

Consider the Exel's inverse semigroup $\mathcal{S}_E(G)$ \cite[Definition 2.1]{exel1998partial}. We know that the partial actions of $G$ on $A$ are in one-to-one correspondence with the actions of $\mathcal{S}_E(G)$ on $A$. Take $\beta$ the action of $\mathcal{S}_E(G)$ on $A$ associated to $\alpha$, that is, given $s = \epsilon_{s_1} \cdots \epsilon_{s_n}[g] \in \mathcal{S}_E(G)$, $E_{s^{-1}} = D_{g^{-1}}D_{g^{-1}s_1} \cdots D_{g^{-1}s_n}$, $E_s = D_gD_{s_1} \cdots D_{s_n}$ and $\beta_s = \alpha_g|_{E_{s^{-1}} \subseteq D_{g^{-1}}}$.

Let's see that $A^{\beta} = A^{\alpha}$. Take $a \in A^{\alpha}$. Then
\begin{align*}
    \alpha_g(a1_g^{-1}) = a1_g, \text{ for all } g \in G.
\end{align*}

Given arbitrary $s = \epsilon_{s_1} \cdots \epsilon_{s_n}[g] \in \mathcal{S}_E(G)$,
\begin{align*}
    \beta_s(a1_{s^{-1}}) & = \alpha_g(a1_{g^{-1}}1_{g^{-1}s_1} \cdots 1_{g^{-1}s_n}) \\ & = \alpha_g(a1_{g^{-1}})\alpha_g(1_{g^{-1}}1_{g^{-1}s_1} \cdots 1_{g^{-1}s_n}) \\ & = a1_g\beta_s(1_{s^{-1}}) \\ & = a1_g1_s = a1_s,
\end{align*}
thus $A^{\alpha} \subseteq A^{\beta}$. Now, take $b \in A^{\beta}$. We have that
\begin{align*}
    \beta_s(b1_{s^{-1}}) = b1_s, \text{ for all } s \in \mathcal{S}_E(G).
\end{align*}

In particular,
\begin{align*}
    \beta_{[g]}(b1_{[g]^{-1}}) = b1_{[g]}, \text{ for all } g \in G.
\end{align*}

But $E_{[g]} = E_g$, $E_{[g]^{-1}} = E_{g^{-1}}$ and $\beta_{[g]} = \alpha_g$. Therefore, $b \in A^{\alpha}$, what allow us to conclude that $A^{\beta} = A^{\alpha}$.

Now we will show that $A$ is a $\beta$-Galois extension of $A^{\beta}$. Consider the $\alpha$-Galois coordinates $x_i, y_i$, $1 \leq i \leq n$, that were presented in the beginning of this example. In this way, with arbitrary $s = \epsilon_{s_1} \cdots \epsilon_{s_n}[g] \in \mathcal{S}_E(G)$,
\begin{align*}
    \sum\limits_{i=1}^{n} x_i\beta_s(y_i1_{s^{-1}}) & = \sum\limits_{i=1}^{n} x_i \alpha_g(y_i1_{g^{-1}}1_{g^{-1}s_1} \cdots 1_{g^{-1}s_n}) \\ & = \sum\limits_{i=1}^{n} x_i \alpha_g(y_i1_{g^{-1}})\alpha_g(1_{g^{-1}}1_{g^{-1}s_1} \cdots 1_{g^{-1}s_n}) \\ & = \left ( \sum\limits_{i=1}^{n} x_i \alpha_g(y_i1_{g^{-1}})  \right ) \alpha_g(1_{g^{-1}}1_{g^{-1}s_1} \cdots 1_{g^{-1}s_n}) \\ & = \delta_{1, g}1_A\beta_s(1_{s^{-1}}) \\ & = \delta_{1,g}1_s.
\end{align*}

Notice now that
\begin{align*}
    \delta_{1,g}1_s = \begin{cases}
    0, \text{ if } s \notin E(\mathcal{S}_E(G)) \\
    1, \text{ if } s \in E(\mathcal{S}_E(G)),
    \end{cases}
\end{align*}
since the idempotents in $\mathcal{S}_E(G)$ are precisely the elements of the form
\begin{align*}
    \epsilon_{s_1} \cdots \epsilon_{s_n}[1] = \epsilon_{s_1} \cdots \epsilon_{s_n}.
\end{align*}

Hence,
\begin{align*}
    \sum\limits_{i=1}^n x_i\beta_s(y_i1_{s^{-1}}) = \delta_{e,s}1_e, \text{ for all } s \in \mathcal{S}_E(G), \text{ } e \in E(\mathcal{S}_E(G)).
\end{align*}

Therefore, every  $\alpha$-partial Galois extension given by a partial action of a group $G$ on an algebra $A$ gives us an example of $\beta$-Galois extension given by the associated action of the Exel's inverse semigroup $\mathcal{S}_E(G)$ on the same algebra $A$. Observe that since $\mathcal{S}_E(G)$ is an inverse monoid, $\beta$ is orthogonal if, and only if, $E_e = \{ 0 \}$, for all $e \in E(\mathcal{S}_E(G))$ such that $e \neq [1_G]$  if, and only if, $D_g = \{ 0 \}$, for all $g \in G$ such that $g \neq 1$.
\end{exe}

\begin{defi}
Let $G$ be an inductive groupoid and $\beta : G \to \mathcal{I}_g(A)$ an orthogonal inductive groupoid action of $G$ on $A$. We say that $A$ is a $\beta$-Galois extension of $A^{\beta}$ if there are $x_i, y_i \in A$, $1 \leq i \leq n$ such that $\sum\limits_{i=1}^n x_i\beta_g(y_i1_{g^{-1}}) = \delta_{e,g}1_e$, for all $g \in G$, $e \in G_0$.
\end{defi}

Since the actions of $\mathbb{G}(S)$ on $A$ are associated with actions of $S$ on $A$ and for an action $\beta : S \to \mathcal{I}_s(A)$, $\beta_s = \mathbb{G}(\beta_s)$, for all $s \in S$, we have the following immediate result. 

\begin{prop}
The functors $\mathbb{S}$ and $\mathbb{G}$ induce a one-to-one correspondence between $\beta$-Galois extensions of inverse semigroups and $\beta$-Galois extensions of inductive groupoids.
\end{prop}

With these considerations, we inherit the following result from the Galois theory for groupoid actions \cite{bagio2012partial}, which characterize Galois extensions.

\begin{theorem}
Let $S$ be an inverse semigroup and $\beta = (\{ E_s \}_{s \in S}, \{ \beta_s : E_{s^{-1}} \rightarrow E_s \}_{s \in S})$ an action of $S$ on $A = \bigoplus\limits_{e \in E(S)} E_e$ where every $E_s$ is unital. The following statements are equivalent:

\emph{(i)} $A$ is a $\beta$-Galois extension of $A^{\beta}$.

\emph{(ii)} $A$ is a projective finitely generated $A^{\beta}$-module and the map $j : A \ltimes_\beta S \rightarrow End_{R^{\beta}}(R)$ given by $j \left ( \sum\limits_{s \in S} a_su_s \right )(x) = \sum\limits_{s \in S} a_s\beta_s(x1_{s^{-1}})$ is an isomorphism of $A$-modules and of $A^{\beta}$-algebras.

\emph{(iii)} For all left $A \ltimes_\beta S$-modules $M$, the map $\mu : A \otimes_{A^{\beta}} M^S \rightarrow M^S$ given by $\mu(a \otimes m) = am$ is an isomorphism of $A$-modules.

\emph{(iv)} The map $\phi : A \otimes_{A^{\beta}} A \rightarrow \prod\limits_{s \in S} E_s$ given by $\phi(a \otimes b) = (a\beta_s(b1_{s^{-1}}))_{s \in S}$ is an isomorphism of $A$-modules.

\emph{(v)} $AtA = A \ltimes_\beta S$, where $t = \sum\limits_{s \in S} 1_su_s$.

\emph{(vi)} The map $\tau' : A \otimes_{A^{\beta}} A \rightarrow A \ltimes_{\beta} S$ given by $\tau'(a \otimes b) = \sum\limits_{s \in S} a\beta_s(b1_{s^{-1}})u_s$ is surjective.

\emph{(vii)} $A$ is a generator of the category of the left $A \ltimes_\beta S$-modules.
\end{theorem}

\begin{proof}
We have that $\mathbb{G}(\beta)$ is an action of $\mathbb{G}(S)$ on $A$, and the extension $A|_{A^{\beta}}$ is $\mathbb{G}(\beta)$-Galois with respect to $\mathbb{G}(S)$. By Corollary \ref{corolarioprodutoscruzados}, $A \ltimes_\beta S \simeq A \ltimes_\beta \mathbb{G}(S)$. Besides that, $A$, $A^{\beta}$, $E_s$ and $\beta_s$ are the same through the transition of $\beta$ to $\mathbb{G}(\beta)$, for all $s \in S$. The result now follows from \cite[Theorem 5.3]{bagio2012partial}.
\end{proof}

Also as a consequence of \cite{bagio2012partial}, we can obtain the following remarks.

\begin{obs}
Since $A$ is a commutative ring, we have that $tr_\beta(A) = A^{\beta}$. In this case notice that

(i) Exists $c \in A$ such that $tr_\beta(c) = 1_A$.

(ii) $A^{\beta}$ is isomorphic to a direct summand of $A$ as $A^{\beta}$-module. In fact, $tr_\beta$ is a surjective map, hence $A \simeq A^{\beta} \oplus \ker (tr_\beta)$.
\end{obs}

Let $B$ be a $A^{\beta}$-subalgebra of $A$. Define $$T_B = \{ s \in S : \beta_s(b1_{s^{-1}}) = b1_s, \text{ for all } b \in B \}.$$ 

Analogously, if $\beta$ is an orthogonal action of a finite groupoid $G$ on $A$ and $B$ is a $A^\beta$-subalgebra of $A$, define
\begin{align} \label{hb}
    H_B = \{ g \in G : \beta_g(b1_{g^{-1}}) = b1_g, \text{ for all } b \in B \}.
\end{align}

\begin{defi}
We say that a subring $B$ of $A$ is $\beta$-strong if for all $s,t \in S$ with $ss^{-1} = tt^{-1}$, $s^{-1}t \notin T_B$ and for any nonzero idempotent $e \in E_s = E_t$, there is an element $b \in B$ such that $\beta_s(b1_{s^{-1}})e \neq \beta_t(b1_{s^{-1}})e$.
\end{defi}

We will begin to construct our way to the Galois correspondence theorem. 


\begin{lemma} \label{lemacompatibilidade}
Given an inverse semigroup $S$ and $\beta$ an orthogonal action of $S$ on $A$, assume that $A$ is $\beta$-Galois over $A^{\beta}$. Let $B$ be a $A^\beta$-subalgebra of $A$. Let $T$ be an inverse subsemigroup of $S$ and $A_T = \bigoplus_{e \in E(T)} E_e$. Define $\beta_T = \{ \beta_t : E_{t^{-1}} \rightarrow E_t, t \in T \}$. Then, denoting by $H = \mathbb{G}(T)$, $A_H = \bigoplus_{e \in H_0} E_e$ and $\beta_H = \{ \beta_h : E_{h^{-1}} \to E_h, h \in H \}$ as in \cite[Theorem 4.1]{paques2018galois}, we have that

\emph{(i)} $A_T = A_H$;

\emph{(ii)} $\beta_T = \mathbb{S}(\beta_H)$;

\emph{(iii)} $T_B = \mathbb{S}(H_B)$.
\end{lemma}

\begin{proof}
(i): It is enough to notice that $E(T) = H_0$.

(ii): $\beta_T = \{ \beta_t : E_{t^{-1}} \rightarrow E_t : t \in T \} = \{ \mathbb{S}(\beta_h) : E_{h^{-1}} \rightarrow E_h : h \in H \} = \mathbb{S}(\beta_H)$. Notice that the compatibility of the actions follows from the discussion in section 3.

(iii): $T_B = \{ s \in S : \beta_s(b1_{s^{-1}}) = b1_s, \text{ for all } b \in B \} = \mathbb{S}(\{ g \in \mathbb{G}(S) : \beta_g(b1_{g^{-1}}) = b1_t, \text{ for all } b \in B \}) = \mathbb{S}(H_B)$.
\end{proof}

\begin{prop}
The set $T_B$ is a full inverse subsemigroup of $S$, for all $A^{\beta}$-subalgebra $B$ of $A$.
\end{prop}
\begin{proof}
We already know that $H_B = \mathbb{G}(T_B)$ is a wide subgroupoid of $\mathbb{G}(S)$. We just need to show that $H_B$ is ordered. Let $x \in H_B$ and $y < x$. Then $\beta_y = \beta_x|_{E_{y^{-1}}} = 0$. So $y \in H_B$. So $H_B$ is a wide ordered subgroupoid of $\mathbb{G}(S)$, from where it follows that $T_B$ is a full inverse subsemigroup of $S$. In particular, this result shows that $(S \setminus \max S) \subseteq T_B$, for all $B$.
\end{proof}

In the groupoid Galois correspondence \cite{paques2018galois} we have the additional condition that the ideals $E_g$ are nonzero, for all  $g \in G$. This motivates the following result.

\begin{prop}
Let $G$ be a groupoid. The following statements are equivalent:
\begin{enumerate}
    \item[(i)] $\max G$ is a subgroupoid of $G$;
    
    \item[(ii)] If $g \in \max G$, then $r(g) \in \max G$;
    
    \item[(iii)] $\max G = \{ g \in G : r(g) \in \max (G_0) \}$.
\end{enumerate}
\end{prop}
\begin{proof}
First notice that if $g \in G$ is such that $r(g) \in \max G_0$, then $g \in \max G$. In fact, if there is $k \in G$ such that $g \leq k$, then $r(g) \leq r(k)$. Since $r(g) \in \max G$, $r(g) = r(k)$. By the uniqueness of restrictions, $g = (r(g) | k) = k$. That guarantee that it always hold $ \{ g \in G : r(g) \in \max (G_0) \} \subseteq \max G$.

It is straightfoward that (i) $\Rightarrow $ (ii). For (ii) $\Rightarrow $ (i), let $g \in \max G$. Suppose that there is $k \in G$ such that $g^{-1} < k$. Then $g < k^{-1}$, which is a contradiction. If $g,h \in \max G$ are such that $\exists gh$, we have that $r(g) = r(gh) \in \max G$, so $gh \in \max G$ by the first observation of this proof. Thus $\max G$ is a groupoid. 

We have clearly that (iii) $\Rightarrow $ (ii). For the reverse implication, assume that $g \in \max G$. Hence, by (ii), $r(g) \in \max G$. In particular, $r(g) \in \max (G_0)$.
\end{proof}

Fix $G$ be a finite ordered groupoid and $\beta$ an orthogonal ordered groupoid action of $G$ on $A$ where $E_g \neq 0$, for all $g \in \max G$. Then we have that $E_{r(g)} = E_g \neq 0$, for all $g \in \max G$. But this implies that $r(g) \in \max G$, since otherwise $E_{r(g)} = 0$. So (ii) of the proposition above holds. Therefore $\max G$ is a groupoid. 

Denote by $\mathsf{SubGr(Max G)}$ the set of all wide subgroupoids of $\max G$ and $\mathsf{SubGr(G)}$ be the set of all wide subgroupoids of $G$ that contain $(G\setminus \max G)$. 

\begin{lemma} \label{lema511}
In the above conditions, there is a bijective correspondence between $\mathsf{SubGr(Max G)}$ and $\mathsf{SubGr(G)}$.
\end{lemma}
\begin{proof}
The map \begin{align*}\mathsf{SubGr(Max G)} & \longrightarrow  \mathsf{SubGr(G)} \\ H' & \longmapsto  H = H' \cup (G\setminus \max G) \end{align*}is clearly bijective, with inverse given by \begin{align*} \mathsf{SubGr(G)} & \longrightarrow  \mathsf{SubGr(Max G)}  \\ H & \longmapsto H \cap \max G.\end{align*} 
\end{proof}

If $G$ is an inductive groupoid and $\beta$ is an orthogonal inductive action of $G$ on $A$ such that $E_g \neq 0$ for all $g \in \max G = G'$, we have that $\beta' = \beta|_{G'} = \{ \{ E_g \}_{g \in G'}, \{ \beta_g \}_{g \in G'} \}$ is an orthogonal action of $G'$ on $A$. Also, $A^\beta = A^{\beta'}$ and a $A^\beta$-subalgebra $B$ of $A$ is $\beta$-strong in the sense of \cite{paques2018galois} if, and only if, it is $\beta'$-strong.

Thus, given a separable, $\beta'$-strong $A^{\beta'}$-subalgebra $B$ of $A$ we can use the Galois correspondence for the case of groupoids \cite[Theorem 4.8(i)]{paques2018galois} to obtain a wide subgroupoid $H'_B$ of $G'$. This is a subgroupoid of $G$, but not wide nor ordered in general. However, we can take the image of $H'_B$ by the map in Lemma \ref{lema511}, adding all the non maximal elements of $G$ on $H'_B$, forming precisely $H_B$ in the sense of (\ref{hb}). Since $\mathbb{S}(H_B) = T_B$ by Lemma \ref{lemacompatibilidade}, it follows that $T_B$ is well-defined. Also, notice that $(S \setminus \max S) \subseteq T_B$, for all $B$.

From now on, suppose that the ideals associated with maximal elements of $E(S)$ are nonzero. The next two results follow from Lemma \ref{lemacompatibilidade} and \cite{paques2018galois}. 

\begin{theorem} \label{teocorrespondencia1}
Let $S$ be an inverse semigroup and $\beta$ an orthogonal action of $S$ on $A$. Assume that $A$ is $\beta$-Galois over $A^{\beta}$. Let $T$ be an inverse subsemigroup of $S$ such that $(S \setminus \max S) \subseteq T$ . Then: 

\emph{(i)} $\beta_T$ is an action of $T$ on $A_T$ and $A_T$ is $\beta_T$-Galois over $B = (A_T)^{\beta_T}$.

\emph{(ii)} $B$ is $A^\beta$-separable.

Besides that, if $T$ is full, then:

\emph{(iii)} $A_T = A$ and $B$ is $\beta$-strong.

\emph{(iv)} $T = T_B$.

\emph{(v)} If $T$ is Clifford and normal, then the quotient inverse semigroup $S/T$ acts on $B$ via an action $\overline{\beta}$ and $B$ is an $\overline{\beta}$-Galois extension of $A^{\beta}$.
\end{theorem}

\begin{proof}
Let $H = \mathbb{G}(T)$ and $H' = H \cap \max \mathbb{G}(S)$. By \cite[Theorem 4.1]{paques2018galois}, $\beta_{H'}$ is an action of $H'$ on $A_{H'} = A_H$ and $A_H$ is $\beta_{H'}$-Galois over $(A_H)^{\beta_{H'}}$. Then $\beta_H$ is an orthogonal inductive action of $H$ on $A_H$ and $A_H$ is $\beta_H$-Galois over $(A_H)^{\beta_{H}}$. We also have that $(A_H)^{\beta_H}$ is $A^{\beta}$-separable. Thus, by Lemma \ref{lemacompatibilidade}, $A_T = A_H$, $\beta_T = \mathbb{S}(\beta_H)$ and the items (i) and (ii) are proven.

Assume that $T$ is full. That means that $H$ is wide, hence $H'$ is wide in $\max \mathbb{G}(S)$ from where it follows item (iii). Since $T_B = \mathbb{S}(H_B)$, the item (iv) follows from Lemma \ref{lema511}.

To show (v), observe that if $T$ is normal in $S$, then $H$ is normal in $\mathbb{G}(S)$, in the same sense of Definition \ref{grpnormal}. Moreover, notice that this definition implies that $H$ is normal in the sense of \cite{paques2018galois}. Therefore the quotient groupoid $\mathbb{G}(S)/\equiv_H$ is well-defined.

Since $T$ is Clifford, $\equiv_H$ and $\sim_H$ are equal, from where it follows that $\mathbb{G}(S)/\equiv_H \mbox{ and } \mathbb{G}(S)/H$ are isomorphic ordered groupoids \cite[Proposition 4.6]{alyamani2018ordered}, hence isomorphic inductive groupoids. However, by \cite[Theorem 4.1(v)]{paques2018galois}, $\max (\mathbb{G}(S))/\equiv_{H'} \ \simeq \mathbb{G}(S)/\equiv_H$ acts on $B$ via an action $\overline{\beta}$. This isomorphism holds since $H$ is wide and $(\mathbb{G}(S) \setminus \max \mathbb{G}(S)) \subseteq H$. Hence $\mathbb{G}(S)/H$ also acts on $B$ via $\overline{\beta}$, where $B$ is $\overline{\beta}$-Galois over $A^{\beta}$. Since, $\mathbb{G}(S)/H \simeq \mathbb{G}(S/T)$, $\mathbb{S}(\overline{\beta})$ is an action of $S/T$ on $B$, provided that $\overline{\beta}$ is an inductive groupoid action. Let us check that $\overline{\beta}$ is, in fact, an ordered groupoid action, since it already is an orthogonal action.

Before proving that, we remind the reader about the definition of $\overline{\beta}$. Given a representative system $\{ g_i \}_{1 \leq i \leq n}$ of $G/\equiv_H$, we will denote by $\overline{g}$ the coset of the element $g$ with relation to $\equiv_H$. Define $e_g = e_{g_i} = tr_{\beta_H}(1_{g_i})$, for all $g \in \overline{g_i}$. Then $\overline{\beta} = (\{ E_{\overline{g}} \}_{\overline{g} \in G/\equiv_H}, \{ \overline{\beta}_{\overline{g}} \}_{\overline{g} \in G/\equiv_H})$ is given by $E_{\overline{g}} = Be_g$ and $\overline{\beta}_{\overline{g}} (be_{g^{-1}}) = \beta_{g_i}(be_{g_i^{-1}})$.

Now we can prove that $\overline{\beta}$ is an ordered groupoid action. Assume that $[g_i]_H \leq_H [g_j]_H$. Thus there are $x, y \in H$ such that $\exists x \cdot g_i \cdot y$ and $x \cdot g_i \cdot y \leq g_j$. Since $\beta$ is an orthogonal inverse semigroup action, $\mathbb{G}(\beta) := \beta$ is as an orthogonal inductive groupoid action of $\mathbb{G}(S)$ on $A$. Therefore, $E_{y^{-1}} = E_{(x \cdot g_i \cdot y)^{-1}} = \delta_{x \cdot g_i \cdot y,g_j}E_{g_j^{-1}}$ and $E_x = E_{x \cdot g_i \cdot y} = \delta_{x \cdot g_i \cdot y, g_j}E_{g_j}$. Besides that, $\beta_{x \cdot g_i \cdot y} = \beta_{g_j}|_{E_{y^{-1}}}$. Without loss of generality we can assume that $x, y \in \{ g_k \}_{1 \leq k \leq n}$. We will denote $\delta = \delta_{x \cdot g_i \cdot y, g_j}$. On the one hand,
\begin{align*}
    \beta_{x \cdot g_i \cdot y}(be_{y^{-1}}) & = \beta_{x \cdot g_i}(\beta_y (be_{y^{-1}})e_{g_i^{-1}}) \\ & = \beta_{x \cdot g_i}(be_{g_i^{-1}}) \\ & = \beta_{x}(\beta_{g_i}(be_{g_i^{-1}})e_{x^{-1}}) \\ & = \beta_{g_i}(be_{g_i^{-1}})e_x \\ & = \beta_{g_i}(be_{g_i^{-1}})e_{x \cdot g_i \cdot y},
\end{align*}
since $x, y \in H = H_B$. On the other hand,
\begin{align*}
    \beta_{x \cdot g_i \cdot y}(be_{y^{-1}}) & = \beta_{g_j}(be_{y^{-1}}) \\ & = \beta_{g_j}(be_{y^{-1}})e_{g_j}.
\end{align*}

Observe that
\begin{align*}
    e_{y^{-1}} & = \sum\limits^{h \in H}_{r(h) = d(y)} 1_h  = \sum\limits^{h \in H}_{r(h) = d(y)} 1_{y^{-1}} \\ & = \sum\limits^{h \in H}_{r(h) = d(y)} \delta 1_{g_j^{-1}}  = \delta \sum\limits^{h \in H}_{r(h) = d(y)} 1_{g_j^{-1}} \\ & = \delta \sum\limits^{h \in H}_{r(h) = d(g_j)} 1_{g_j^{-1}} = \delta e_{g_j^{-1}}.
\end{align*}

And we can see that $e_x = \delta e_{g_j}$ following the same steps. Since
\begin{align*}
    \beta_{g_i}(be_{g_i^{-1}})e_x = \beta_{g_j}(be_{y^{-1}})e_{g_j},
\end{align*}
we have that
\begin{align} \label{ig2}
    \delta \beta_{g_i}(be_{g_i^{-1}})e_{g_j} = \delta \beta_{g_j}(be_{g_j^{-1}})e_{g_j}.
\end{align}

Since $T$ is Clifford, $r(x) = d(x) = r(g_i)$, from where it follows that $e_x = e_{x^{-1}} = e_{g_i}$. In this case,
\begin{align} \label{ig3}
    E_{\overline{g_i}} = Be_{g_i} = Be_x = \delta Be_{g_j} = \delta E_{\overline{g_j}}.
\end{align}

Hence (\ref{ig2}) and (\ref{ig3}) tell us that $[g_i]_H \leq_H [g_j]_H$ implies $E_{\overline{g_i}} \subseteq E_{\overline{g_j}}$ and $\overline{\beta}_{\overline{g_i}} = \overline{\beta}_{\overline{g_j}} | _{E_{\overline{g_i}^{-1}}}$. That is, $\overline{\beta}$ is an ordered groupoid action as we desired.

\end{proof}

\begin{theorem} \label{teocorrespondencia2}
Assume that $A$ is $\beta$-Galois over $A^{\beta}$ and let $B$ be a separable, $\beta$-strong, $A^{\beta}$-subalgebra of $A$. Then denoting $T := T_B$ we have that $A^{\beta_T} = B$.
\end{theorem}

\begin{proof}

By \cite[Theorem 4.5]{paques2018galois}, the result holds for $H = \mathbb{G}(T)$. By Lemma \ref{lemacompatibilidade}, $\beta_T = \mathbb{S}(\beta_H)$. Thus $A^{\beta_T} = A^{\beta_H} = B$.

\end{proof}

Now we can state the main result of this work, the Fundamental Theorem of the Galois theory, which extends \cite[Theorem 2.3]{chase1969galois}.

\begin{theorem} \label{teocorr} Let $S$ be an inverse semigroup and $\beta$ an orthogonal action of $S$ on $A$. Assume that $A$ is $\beta$-Galois over $A^{\beta}$.\begin{itemize}

 \item[(a)] \textbf{The Galois Correspondence:} there is a one-to-one correspondence between the full inverse subsemigroups $T$ of $S$ such that $(S \setminus \max S) \subseteq T$ and the separable, $\beta$-strong, $A^{\beta}$-subalgebras $B$ of $A$. More precisely, under these conditions, the maps $T \mapsto A^{\beta_T}$ and $B \mapsto T_B$ are inverses, in the sense that
\begin{align*}
    A^{\beta_{T_B}} = B \text{ and } T_{A^{\beta_T}} = T.
\end{align*}

\item[(b)] Furthermore, if $T$ is a normal Clifford inverse subsemigroup of $S$ such that $(S \setminus \max S) \subseteq T$, then $S/T$ acts on $A^{\beta_T}$ via an orthogonal action $\overline{\beta}$ and $A^{\beta_T}$ is a $\overline{\beta}$-Galois extension of $A^{\beta}$.
\end{itemize}
\end{theorem}

\begin{proof}
The proof follows from Theorems \ref{teocorrespondencia1} and \ref{teocorrespondencia2}.
\end{proof}

We give a graphic illustration of Theorem \ref{teocorr}:

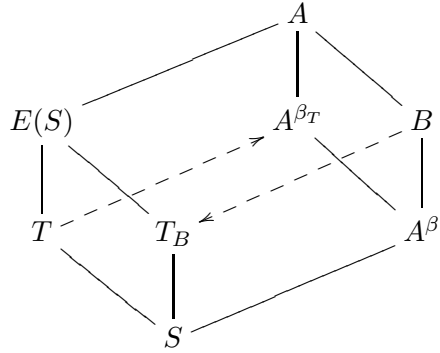
$$ \xymatrix { &  & A \ar@{-}[dll] \ar@{-}[d]\ar@{-}[dr] & \\
E(S) \ar@{-}[d] \ar@{-}[dr] & & A^{\beta_T} \ar@{<--}[dll]\ar@{-}[dr] & B \ar@{-}[d]\\
T\ar@{-}[dr] & T_B \ar@{-}[d]\ar@{<--}[rru] & & A^{\beta} \ar@{-}[dll] \\
 & S &  & }  $$

We close this paper with an example of the Galois correspondence. Before that, we present the notion of connected groupoid. We say that a groupoid $G$ is \emph{connected} when for all $e,f \in G_0$ there is $g \in G$ such that $d(g) = e$ and $r(g) = f$. It is well known that every connected groupoid $G$ can be written as $G \simeq G_0^2 \times H$, where $H$ is a group - called the isotropy group of $G$ - and $G_0^2$ is the coarse groupoid associated with $G_0$, that is, the groupoid with elements $(e,f) \in G_0 \times G_0$ where $d(e,f) = e$, $r(e,f) = f$, $\exists (e,f)(i,j)$ if, and only if, $ j = e$ and in this case $(e,f)(i,j) = (i,f)$. It is clear that every groupoid is a disjoint union of connected groupoids.

 \begin{exe} \label{exfinal}
  
  Let $X = \{ 1, 2, 3 \}$. We first construct $I(X)$, the inverse semigroup of partial bijections of $X$. This set has 34 elements, namely $I_{\emptyset}, I_1, I_2, I_3, I_{12}, $ $ I_{13},$ $ I_{23}, $ $ I_{123}, T_{12},$ $ T_{21}, $ $ T_{13},$ $ T_{31}, T_{23}, T_{32}, S_{12}, S_{13}, S_{23}, D_{12}^{13},$ $ D_{13}^{12}, $ $ D_{12}^{23},$ $ D_{23}^{12}, $ $ D_{13}^{23}, $ $ D_{23}^{13}, $ $ P_{12}^{13}, $ $ P_{13}^{12}, P_{12}^{23}, P_{23}^{12}, P_{13}^{23}, P_{23}^{13}, T_{12}^3, T_{13}^2, T_{23}^1,$ $ S_{123},$ $ S_{132}$, defined by:
  
\begin{minipage}{7cm}
\begin{align*}
    T_{ij} : \{ i \} & \rightarrow \{ j \} \\
    i & \mapsto j
\end{align*}
\end{minipage}
\begin{minipage}{3cm}
\begin{align*}
    S_{ij} : \{i, j\} & \rightarrow \{ i, j\} \\
    i & \mapsto j \\
    j & \mapsto i
\end{align*}
\end{minipage}

\vspace{0,4cm}

\begin{minipage}{7cm}
\begin{align*}
    D_{ij}^{jk} : \{i, j\} & \rightarrow \{ j, k\} \\
    i & \mapsto k \\
    j & \mapsto j
\end{align*}
\end{minipage}
\begin{minipage}{3cm}
\begin{align*}
    P_{ij}^{jk} : \{i, j\} & \rightarrow \{ j, k\} \\
    i & \mapsto j \\
    j & \mapsto k
\end{align*}
\end{minipage}

\vspace{0,3cm}

A good way to distinguish between the $D$'s and the $P$'s is to remember that there is always a unique element in the intersection of the domain and the range of these bijections. The $D$-type bijections fix this element, while the $P$-type bijections do not.

\begin{minipage}{6.5cm}
\begin{align*}
    S_{ijk} : \{i, j, k\} & \rightarrow \{ i, j, k\} \\
    i & \mapsto j \\
    j & \mapsto k \\
    k & \mapsto i
\end{align*}
\end{minipage}
\begin{minipage}{4cm}
\begin{align*}
    T_{ij}^k : \{i, j, k \} & \rightarrow \{ i, j, k \} \\
    i & \mapsto j \\
    j & \mapsto i \\
    k & \mapsto k
\end{align*}
\end{minipage}

\vspace{0,3cm}

The $I$-type maps are the identities and then compose $E(I(X))$. It is easy to see that $\{ I_{123}, T_{12}^3, T_{13}^2, T_{23}^1, S_{123}, S_{132} \} \simeq D_3$, the dihedral group of degree 3. We will identify these two sets. The meet semilattice $E(I(X))$ has the following order diagram:

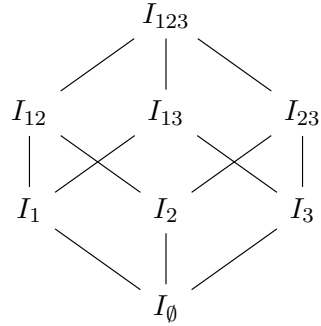
\begin{center}
\begin{tikzcd}
               & I_{123} \arrow[ld, dash] \arrow[d, dash] \arrow[rd, dash] & \\
I_{12} \arrow[d, dash] \arrow[rd, dash] & I_{13} \arrow[ld, dash] \arrow[rd, dash] & I_{23} \arrow[ld, dash] \arrow[d, dash] \\
               I_{1} \arrow[rd, dash] & I_{2} \arrow[d, dash] & I_{3} \arrow[ld, dash] \\
               & I_{\emptyset} &
\end{tikzcd}
\end{center}

The order diagram of $I(X)$ can be split into two parts:

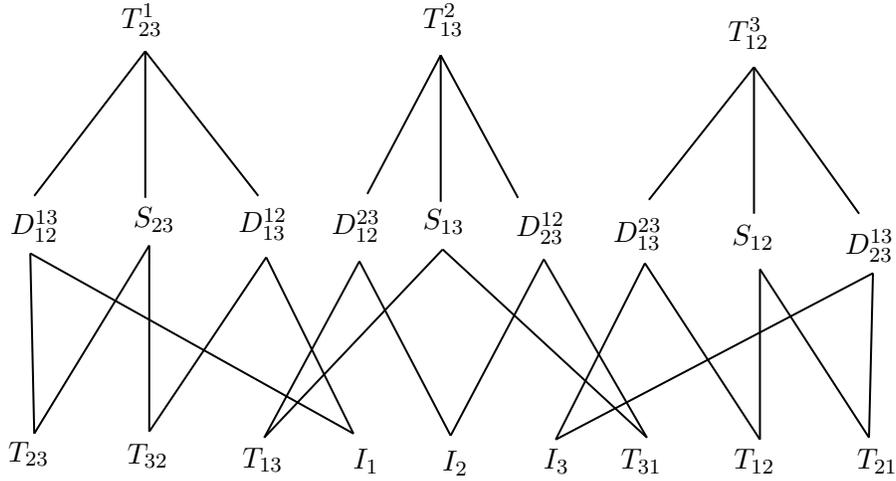
\begin{center}

\tikzset{every picture/.style={line width=0.75pt}} 

\begin{tikzpicture}[x=0.75pt,y=0.75pt,yscale=-1,xscale=1]

\draw    (117.5,226) -- (119.5,317) ;
\draw    (117.5,226) -- (280.5,317) ;
\draw    (236.5,228) -- (177.5,316) ;
\draw    (236.5,228) -- (280.5,317) ;
\draw    (177.5,222) -- (119.5,317) ;
\draw    (177.5,222) -- (177.5,316) ;
\draw    (283.5,230) -- (329.5,318) ;
\draw    (283.5,230) -- (235.5,319) ;
\draw    (325.5,224) -- (235.5,319) ;
\draw    (325.5,224) -- (428.5,319) ;
\draw    (376.5,229) -- (329.5,318) ;
\draw    (376.5,229) -- (428.5,319) ;
\draw    (427.5,231) -- (382.5,320) ;
\draw    (427.5,231) -- (485.5,320) ;
\draw    (542.5,236) -- (382.5,320) ;
\draw    (485.5,234) -- (540.5,319) ;
\draw    (485.5,234) -- (485.5,320) ;
\draw    (542.5,236) -- (540.5,319) ;
\draw    (175.5,124) -- (175.5,198) ;
\draw    (175.5,124) -- (119.5,197) ;
\draw    (175.5,124) -- (232.5,197) ;
\draw    (324.5,126) -- (324.5,200) ;
\draw    (324.5,126) -- (287.5,196) ;
\draw    (324.5,126) -- (363.5,198) ;
\draw    (482.5,132) -- (482.5,206) ;
\draw    (482.5,132) -- (429.5,199) ;
\draw    (482.5,132) -- (535.5,206) ;

\draw (279,323.4) node [anchor=north west][inner sep=0.75pt]    {$I_{1}$};
\draw (324,324.4) node [anchor=north west][inner sep=0.75pt]    {$I_{2}$};
\draw (375,323.4) node [anchor=north west][inner sep=0.75pt]    {$I_{3}$};
\draw (105,319.4) node [anchor=north west][inner sep=0.75pt]    {$T_{23}$};
\draw (165,320.4) node [anchor=north west][inner sep=0.75pt]    {$T_{32}$};
\draw (223,322.4) node [anchor=north west][inner sep=0.75pt]    {$T_{13}$};
\draw (414,323.4) node [anchor=north west][inner sep=0.75pt]    {$T_{31}$};
\draw (471,323.4) node [anchor=north west][inner sep=0.75pt]    {$T_{12}$};
\draw (533,323.4) node [anchor=north west][inner sep=0.75pt]    {$T_{21}$};
\draw (106,202.4) node [anchor=north west][inner sep=0.75pt]    {$D^{13}_{12}$};
\draw (221,201.4) node [anchor=north west][inner sep=0.75pt]    {$D^{12}_{13}$};
\draw (268,202.4) node [anchor=north west][inner sep=0.75pt]    {$D^{23}_{12}$};
\draw (168,201.4) node [anchor=north west][inner sep=0.75pt]    {$S_{23}$};
\draw (361,203.4) node [anchor=north west][inner sep=0.75pt]    {$D^{12}_{23}$};
\draw (410,206.4) node [anchor=north west][inner sep=0.75pt]    {$D^{23}_{13}$};
\draw (527,212.4) node [anchor=north west][inner sep=0.75pt]    {$D^{13}_{23}$};
\draw (470,210.4) node [anchor=north west][inner sep=0.75pt]    {$S_{12}$};
\draw (314,202.4) node [anchor=north west][inner sep=0.75pt]    {$S_{13}$};
\draw (468,105.4) node [anchor=north west][inner sep=0.75pt]    {$T^{3}_{12}$};
\draw (314,98.4) node [anchor=north west][inner sep=0.75pt]    {$T^{2}_{13}$};
\draw (162,98.4) node [anchor=north west][inner sep=0.75pt]    {$T^{1}_{23}$};

\end{tikzpicture}
\end{center}

\hfill

\noindent and

\hfill

\begin{center}

\tikzset{every picture/.style={line width=0.75pt}} 

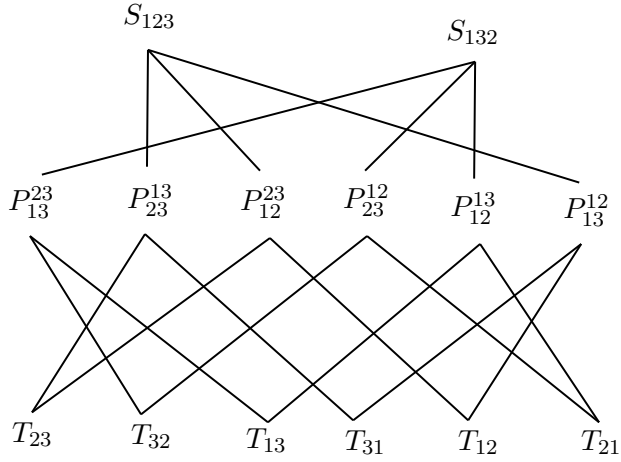
\begin{tikzpicture}[x=0.75pt,y=0.75pt,yscale=-1,xscale=1]

\draw    (239.5,228) -- (182.5,318) ;
\draw    (181.5,229) -- (237.5,319) ;
\draw    (181.5,229) -- (301.5,323) ;
\draw    (239.5,228) -- (344.5,322) ;
\draw    (302.5,230) -- (402.5,322) ;
\draw    (302.5,230) -- (182.5,318) ;
\draw    (351.5,229) -- (468.5,323) ;
\draw    (351.5,229) -- (237.5,319) ;
\draw    (408.5,233) -- (468.5,323) ;
\draw    (408.5,233) -- (301.5,323) ;
\draw    (459.5,233) -- (402.5,322) ;
\draw    (459.5,233) -- (344.5,322) ;
\draw    (241,135) -- (240.5,194) ;
\draw    (241,135) -- (297.5,196) ;
\draw    (241,135) -- (458.5,202) ;
\draw    (406,141) -- (405.5,200) ;
\draw    (406,141) -- (350.5,196) ;
\draw    (406,141) -- (187.5,198) ;

\draw (171,321.4) node [anchor=north west][inner sep=0.75pt]    {$T_{23}$};
\draw (231,322.4) node [anchor=north west][inner sep=0.75pt]    {$T_{32}$};
\draw (289,324.4) node [anchor=north west][inner sep=0.75pt]    {$T_{13}$};
\draw (339,325.4) node [anchor=north west][inner sep=0.75pt]    {$T_{31}$};
\draw (396,325.4) node [anchor=north west][inner sep=0.75pt]    {$T_{12}$};
\draw (458,325.4) node [anchor=north west][inner sep=0.75pt]    {$T_{21}$};
\draw (169,202.4) node [anchor=north west][inner sep=0.75pt]    {$P^{23}_{13}$};
\draw (229,201.4) node [anchor=north west][inner sep=0.75pt]    {$P^{13}_{23}$};
\draw (286,202.4) node [anchor=north west][inner sep=0.75pt]    {$P^{23}_{12}$};
\draw (227,110.4) node [anchor=north west][inner sep=0.75pt]    {$S_{123}$};
\draw (338,201.4) node [anchor=north west][inner sep=0.75pt]    {$P^{12}_{23}$};
\draw (392,205.4) node [anchor=north west][inner sep=0.75pt]    {$P^{13}_{12}$};
\draw (449,206.4) node [anchor=north west][inner sep=0.75pt]    {$P^{12}_{13}$};
\draw (390,118.4) node [anchor=north west][inner sep=0.75pt]    {$S_{132}$};

\end{tikzpicture}

\end{center}

Since $I_{123} = \max \{ E(I(X)) \}$, we have that every action of $I(X)$ on an algebra $A$ where (1) holds is, in fact, an action of $D_3$ in $A$. 

Consider the inverse subsemigroup of $I(X)$ given by $S = I(X) - D_3$. We will construct an action of $S$ on an $R$-algebra $A = \bigoplus_{i = 1}^6 Re_i$, where $R$ is a commutative ring, $e_ie_j = \delta_{i,j}e_i$ and $\sum_{i = 1}^6 e_i = 1_A$.

Define $E_{I_{12}} = Re_1 \oplus Re_2 $, $E_{I_{13}} = Re_3 \oplus Re_4$ and $E_{I_{23}} = Re_5 \oplus Re_6$. Hence $A = E_{I_{12}} \oplus E_{I_{13}} \oplus E_{I_{23}} = \bigoplus\limits_{e \in E(S)} E_e$.

The isomorphisms are given by

\noindent\begin{minipage}{4,2cm}
\begin{align*}
    \beta_{D_{12}^{13}} : E_{I_{12}} & \rightarrow E_{I_{13}} \\
    ae_1 + be_2 & \mapsto ae_3 + be_4
\end{align*}
\end{minipage}
\begin{minipage}{4cm}
\begin{align*}
    \beta_{D_{12}^{23}} : E_{I_{12}} & \rightarrow E_{I_{23}} \\
    ae_1 + be_2 & \mapsto be_5 + ae_6
\end{align*}
\end{minipage}
\begin{minipage}{4,2cm}
\begin{align*}
    \beta_{D_{13}^{23}} : E_{I_{13}} & \rightarrow E_{I_{23}} \\
    ae_3 + be_4 & \mapsto ae_5 + be_6
\end{align*}
\end{minipage}

\noindent\begin{minipage}{4,2cm}
\begin{align*}
    \beta_{P_{12}^{13}} : E_{I_{12}} & \rightarrow E_{I_{13}} \\
    ae_1 + be_2 & \mapsto be_3 + ae_4
\end{align*}
\end{minipage}
\begin{minipage}{4,2cm}
\begin{align*}
    \beta_{P_{12}^{23}} : E_{I_{12}} & \rightarrow E_{I_{23}} \\
    ae_1 + be_2 & \mapsto ae_5 + be_6
\end{align*}
\end{minipage}
\begin{minipage}{4cm}
\begin{align*}
    \beta_{P_{13}^{23}} : E_{I_{13}} & \rightarrow E_{I_{23}} \\
    ae_3 + be_4 & \mapsto be_5 + ae_6
\end{align*}
\end{minipage}

\noindent\begin{minipage}{4,2cm}
\begin{align*}
    \beta_{S_{12}} : E_{I_{12}} & \rightarrow E_{I_{12}} \\
    ae_1 + be_2 & \mapsto be_1 + ae_2
\end{align*}
\end{minipage}
\begin{minipage}{4,2cm}
\begin{align*}
    \beta_{S_{13}} : E_{I_{13}} & \rightarrow E_{I_{13}} \\
    ae_3 + be_4 & \mapsto be_3 + ae_4
\end{align*}
\end{minipage}
\begin{minipage}{4cm}
\begin{align*}
\beta_{S_{23}} : E_{I_{23}} & \rightarrow E_{I_{23}} \\
    ae_5 + be_6 & \mapsto be_5 + ae_6    
\end{align*}
\end{minipage}

\vspace{0,3cm}
\noindent and the inverses are obvious. Notice that in our notation $D$ means fix and $P$ means change, but in the action's isomorphisms between $E_{I_{13}}$ and $E_{I_{23}}$ this definition is switched. Thus $\beta = ( \{ E_s \}_{s \in S}, \{ \beta_s \}_{s \in S})$ is an action of $S$ on $A$ and $A^{\beta} = R1_A \simeq R$. Moreover, $A$ is $\beta$-Galois over $R$ with system of $\beta$-Galois coordinates equals to $\{ x_i = y_i = e_i\}_{i = 1, \ldots , 6}$. These are the only ideals that matter to us, since $\max S = \{ I_{12}, I_{13}, I_{23}, S_{12}, S_{13}, S_{23}, D_{12}^{13}, D_{13}^{12}, D_{12}^{23}, $ $ D_{23}^{12}, $ $ D_{13}^{23}, $ $ D_{23}^{13}, P_{12}^{13}, $ $ P_{13}^{12}, P_{12}^{23}, P_{23}^{12}, P_{13}^{23}, P_{23}^{13} \}$.

We will compute the $A^{\beta}$-subalgebras between $R$ and $A$ that are separable and $\beta$-strong, as well as the associated inverse subsemigroups. Consider $T = (S \setminus \max S) \cup E(S) =  \{ I_{\emptyset}, I_1, I_2, I_3, I_{12}, I_{13},$  $ I_{23}, T_{12}, T_{21}, T_{13}, T_{31}, T_{23}, T_{32} \}$. We list the 31 $\beta$-strong algebras and their associated inverse subsemigroups:

\hfill

\noindent $A \leftrightarrow T_A = T$

\noindent $B_1 = R(e_1 + e_2) \oplus Re_3 \oplus Re_4 \oplus Re_5 \oplus Re_6 \leftrightarrow T_{B_1} = T \cup \{ S_{12} \}$

\noindent $B_{2} = Re_1 \oplus Re_2 \oplus R(e_3 + e_4) \oplus  Re_5 \oplus Re_6 \leftrightarrow T_{B_2} = T \cup \{ S_{13} \}$

\noindent $B_{3} = Re_1 \oplus Re_2 \oplus  Re_3 \oplus Re_4 \oplus R(e_5 + e_6) \leftrightarrow T_{B_3} = T \cup \{ S_{23} \}$

\noindent $B_{4} = R(e_1 + e_3) \oplus R(e_2 + e_4) \oplus Re_5 \oplus Re_6 \leftrightarrow T_{B_4} = T \cup \{ D_{12}^{13}, D_{13}^{12} \}$

\noindent $B_{5} = R(e_1 + e_6) \oplus R(e_2 + e_5) \oplus Re_3 \oplus Re_4 \leftrightarrow T_{B_5} = T \cup \{ D_{12}^{23}, D_{23}^{12} \}$

\noindent $B_{6} = Re_1 \oplus Re_2 \oplus R(e_3 + e_5) \oplus R(e_4 + e_6)  \leftrightarrow T_{B_6} = T \cup \{ D_{13}^{23}, D_{23}^{13} \}$

\noindent $B_{7} = R(e_1 + e_4) \oplus R(e_2 + e_3) \oplus Re_5 \oplus Re_6 \leftrightarrow  T_{B_7} = T \cup \{ P_{12}^{13}, P_{13}^{12} \}$

\noindent $B_{8} = R(e_1 + e_5) \oplus R(e_2 + e_6) \oplus Re_3 \oplus Re_4 \leftrightarrow T_{B_8} = T \cup \{ P_{12}^{23}, P_{23}^{12} \}$

\noindent $B_{9} = Re_1 \oplus Re_2 \oplus R(e_3 + e_6) \oplus R(e_4 + e_5)  \leftrightarrow T_{B_9} = T \cup \{ P_{13}^{23}, P_{23}^{13} \}$

\noindent $C_{1} = R(e_1 + e_2) \oplus R(e_3 + e_4) \oplus Re_5 \oplus Re_6 \leftrightarrow T_{C_1} = T \cup \{ S_{12}, S_{13} \}$

\noindent $C_{2} = R(e_1 + e_2) \oplus Re_3 \oplus Re_4 \oplus R(e_5 + e_6) \leftrightarrow T_{C_2} = T \cup \{ S_{12}, S_{23} \}$

\noindent $C_{3} = Re_1 \oplus Re_2 \oplus R(e_3 + e_4) \oplus R(e_5 + e_6)\leftrightarrow T_{C_3} = T \cup \{ S_{13}, S_{23} \}$

\noindent $C_{4} = R(e_1 + e_3) \oplus R(e_2 + e_4) \oplus R(e_5 + e_6) \leftrightarrow T_{C_4} = T \cup \{ S_{23}, D_{12}^{13}, D_{13}^{12} \}$

\noindent $C_{5} = R(e_1 + e_6) \oplus R(e_2 + e_5) \oplus R(e_3 + e_4) \leftrightarrow T_{C_5} = T \cup \{S_{13}, D_{12}^{23}, D_{23}^{12} \}$

\noindent $C_{6} = R(e_1 + e_2) \oplus R(e_3 + e_5) \oplus R(e_4 + e_6) \leftrightarrow T_{C_6} = T \cup \{ S_{12}, D_{13}^{23}, D_{23}^{13} \}$

\noindent $C_{7} = R(e_1 + e_4) \oplus R(e_2 + e_3) \oplus R(e_5 + e_6) \leftrightarrow T_{C_7} = T \cup \{S_{23}, P_{12}^{13}, P_{13}^{12} \}$

\noindent $C_{8} = R(e_1 + e_5) \oplus R(e_2 + e_6) \oplus R(e_3 + e_4) \leftrightarrow T_{C_8} = T \cup \{ S_{13}, P_{12}^{23}, P_{23}^{12} \}$

\noindent $C_{9} = R(e_1 + e_2) \oplus R(e_3 + e_6) \oplus R(e_4 + e_5) \leftrightarrow T_{C_9} = T \cup \{ S_{12}, P_{13}^{23}, P_{23}^{13} \}$

\noindent $C_{10} = R(e_1 + e_3 + e_5) \oplus R(e_2 + e_4 + e_6) \leftrightarrow T_{C_{10}} = T_{B_4} \cup T_{B_6} \cup T_{B_8}$

\noindent $C_{11} = R(e_1 + e_3 + e_6) \oplus R(e_2 + e_4 + e_5) \leftrightarrow T_{C_{11}} = T_{B_4} \cup T_{B_5} \cup T_{B_9}$

\noindent $C_{12} = R(e_1 + e_4 + e_5) \oplus R(e_2 + e_3 + e_6) \leftrightarrow T_{C_{12}} =  T_{B_7} \cup T_{B_8} \cup T_{B_9}$

\noindent $C_{13} = R(e_1 + e_4 + e_5) \oplus R(e_2 + e_3 + e_5) \leftrightarrow T_{C_{13}} = T_{B_5} \cup T_{B_6} \cup T_{B_7}$

\noindent $F_{1} = R(e_1 + e_2 + e_3 + e_4) \oplus Re_5 \oplus Re_6 \leftrightarrow T_{F_1} = T_{C_1} \cup T_{B_4} \cup T_{B_7}$

\noindent $F_{2} = R(e_1 + e_2 + e_5 + e_6) \oplus Re_3 \oplus Re_4 \leftrightarrow T_{F_2} = T_{C_2} \cup T_{B_5} \cup T_{B_8}$

\noindent $F_{3} = Re_1 \oplus Re_2 \oplus R(e_3 + e_4 + e_5 + e_6) \leftrightarrow T_{F_3} = T_{C_3} \cup T_{B_6} \cup T_{B_9}$

\noindent $F_4 = R(e_1 + e_2) \oplus R(e_3 + e_4) \oplus R(e_5 + e_6) \leftrightarrow T_{F_4} = T_{C_1}\cup T_{C_2}$

\noindent $J_{1} = R(e_1 + e_2 + e_3 + e_4) \oplus R(e_5 + e_6) \leftrightarrow T_{J_1} = T_{F_1}\cup T_{F_4}$

\noindent $J_{2} = R(e_1 + e_2 + e_5 + e_6) \oplus R(e_3 + e_4) \leftrightarrow T_{J_2} = T_{F_2}\cup T_{F_4}$

\noindent $J_{3} = R(e_1 + e_2) \oplus R(e_3 + e_4 + e_5 + e_6) \leftrightarrow T_{J_3} = T_{F_3}\cup T_{F_4}$

\noindent $R \leftrightarrow T_R = S$.

Consider $U = T_{J_1}$. When seen as a groupoid, it has four connected components: $G_{\emptyset} = \{I_\emptyset \}$, $G_1 = T \setminus \{ I_{\emptyset}, I_{12}, I_{13}, I_{23} \} \simeq \{ I_1, I_2, I_3\}^2$, the coarse groupoid associated with $\{ I_1, I_2, I_3 \}$, $G_{12,13} = \{ I_{12}, I_{13}, S_{12}, S_{13} D_{12}^{13}, D_{13}^{12},$ $ P_{12}^{13}, P_{13}^{12} \} \simeq \{ I_{12}, I_{13}\}^2 \times \mathbb{Z}_2$ and $G_{23} = \{ I_{23}, S_{23} \} \simeq \mathbb{Z}_2$. Since $U$ is a full inverse subsemigroup of $S$, we have that $\beta_{U}$ is an action of $U$ on $A_{U} = A$ and $A$ is $\beta_{U}$-Galois over $(A_{U})^{\beta_{U}} = J_1$. Notice that in this case the groupoid $\max \mathbb{G}(U) = G_{12,13} \cup G_{23}$ is not connected, in contrast to Example \ref{exgeral} in where $\max \mathbb{G}(S)$ is always connected. 
 \end{exe}

\bibliographystyle{amsalpha}
{
}

\end{document}